\documentclass[10pt,a4paper,twoside,reqno]{amsart}
\usepackage{amssymb}
\usepackage{MnSymbol}
\usepackage[utf8]{inputenc}
\usepackage{color}
\usepackage[inline]{enumitem}
\usepackage{tikz}
\usetikzlibrary{positioning,decorations.pathreplacing,decorations.markings,arrows.meta,calc,fit,quotes,cd,math,arrows,backgrounds,shapes.geometric}

\usepackage[a4paper,top=3cm,bottom=3cm,inner=3cm,outer=3cm]{geometry}
\usepackage{graphicx}
\usepackage{todonotes}
\usepackage[backend=biber,style=alphabetic,giveninits=true,maxnames=5]{biblatex}
\definecolor{darkgreen}{rgb}{0,0.45,0}
\usepackage{bm}
\usepackage[colorlinks,citecolor=darkgreen,urlcolor=gray,final,hyperindex,linktoc=page,hyperfootnotes=false]{hyperref}
\usepackage[capitalize]{cleveref}
\crefname{equation}{}{}
\crefname{thm}{Theorem}{Theorems}
\crefname{prop}{Proposition}{Propositions}

\theoremstyle{plain}
\newtheorem{thm}{Theorem}
\newtheorem{cor}[thm]{Corollary}
\newtheorem{lemma}[thm]{Lemma}
\newtheorem{prop}[thm]{Proposition}

\theoremstyle{remark}

\newtheorem{rmk}[thm]{Remark}

\theoremstyle{definition}

\newtheorem{df}[thm]{Definition}

\definecolor{mypurple}{rgb}{0.5, 0.0, 0.5}

\newcommand{\conv}{*}
\newcommand{\lhom}[2]{[#1,#2]_{\ell}}
\newcommand{\rhom}[2]{[#1,#2]_{r}}
\newcommand{\enr}[1]{\mathsf{#1}}
\newcommand{\ca}{\mathcal}
\newcommand{\nc}{\mathsf}
\newcommand{\B}{\mathbf}
\newcommand{\mi}{\textrm{-}}
\newcommand{\Set}{\nc{Set}}
\newcommand{\Cat}{\nc{Cat}}
\newcommand{\Vect}{\nc{Vect}}

\newcommand{\Sp}{\nc{Sp}}
\newcommand{\Grd}{\nc{Grd}}

\newcommand{\VSym}{\nc{Sp}(\ca{V})}
\newcommand{\VSp}{\nc{Sp}(\ca{V})}
\newcommand{\VP}{\ca{V}^\ca{P}}
\newcommand{\VGrd}{\nc{Grd}(\ca{V})}
\newcommand{\VOpd}{\nc{Opd}(\ca{V})}
\newcommand{\VCoopd}{\nc{Coopd}(\ca{V})}
\newcommand{\VOpdp}{\nc{Opd}_+(\ca{V})}
\newcommand{\VCoopdp}{\nc{Coopd}_+(\ca{V})}

\newcommand{\wt}{\widetilde}
\newcommand{\1}{\mathbf{1}}

\newcommand{\ot}{\otimes}
\newcommand{\Mon}{\nc{Mon}}
\newcommand{\Comon}{\nc{Comon}}

\newcommand{\Coalg}{\nc{Coalg}}
\newcommand{\I}{\mathbf{I}}
\newcommand{\J}{\mathbf{X}}
\newcommand{\X}{\mathbf{X}}
\newcommand{\VN}{\ca{V}^\ca{N}}
\newcommand{\op}{\mathrm{op}}
\newcommand{\id}{\mathrm{id}}
\newcommand{\Hom}{\mathrm{Hom}}
\DeclareMathOperator*\colim{colim}
\newcommand{\HOM}{\textrm{\scshape{Hom}}}
\tikzset{tick/.style={postaction={decorate,decoration={markings,mark=at position 0.5 with {\draw[-] (0,.4ex) -- (0,-.4ex);}}}}}

\begin{document}
\title{Duoidal categories, measuring comonoids and enrichment}

\author{Ignacio L\'{o}pez Franco}
\address{Departamento de Matem\'atica y Aplicaciones\\ CURE, Universidad de la
  Rep\'ublica\\Tacuaremb\'o s/n. Maldonado, Uruguay.}
\email{ilopez@cure.edu.uy}

\author{Christina Vasilakopoulou}
\address{Department of Mathematics, University of Patras, Greece}
\email{cvasilak@math.upatras.gr}

\begin{abstract}
  We extend the theory of Sweeder's measuring comonoids to the framework of
  duoidal categories: categories equipped with two compatible monoidal
  structures. We use one of the tensor products to endow the category of monoids
  for the other with an enrichment in the category of comonoids. The enriched
  homs are provided by the universal measuring comonoids. We study a number of
  duoidal structures on categories of graded objects and of species and the
  associated enriched categories, such as an enrichment of graded (twisted)
  monoids in graded (twisted) comonoids, as well as two enrichments of symmetric operads in
  symmetric cooperads.
\end{abstract}

\maketitle

\setcounter{tocdepth}{1}
\tableofcontents

\section{Introduction}

This paper ties the theories of duoidal categories and enriched
categories arising from universal measuring coalgebras. Both theories have seen
recent developments, as we set out below, but were previously unconnected.

Measuring coalgebras were introduced by M.~E.~Sweedler~\cite{Sweedler} as a
solution to the inconvenience that the dual vector space of an algebra need not
be a coalgebra. The universal measuring coalgebra $P(A,B)$ of a pair of algebras
$A$ and $B$ is characterised by the property that coalgebra morphisms
$C\to P(A,B)$ classify algebra morphisms $A\to\mathrm{Hom}(C,B)$. The coalgebra
$A^\circ=P(A,k)$, usually named the \emph{Sweedler dual} or \emph{finite dual}
of the algebra $A$ over the field $k$, plays an important role in the theory of
Hopf algebras.

Universal measuring coalgebras endow the category of algebras with an enrichment
in the category of coalgebras, an observation first made by G.~Wraith in
1968. The coalgebra $P(A,B)$ is the enriched hom from $A$ to $B$, from which one
can recover all the algebra morphisms $A\to B$ in the form of coalgebra
morphisms $k\to P(A,B)$, i.e., group-like elements. This hom, thus, has much
more information than just the algebra morphisms, for example, information about
derivations~\cite{Differenceoperators}.  Several authors have noticed that
vector spaces are unnecessary to construct universal measuring
coalgebras. Modules over a commutative Noetherian ring are used
in~\cite{MR1780737}, while \cite{AnelJoyal} uses differential-graded vector
spaces. In the latter, universal measuring coalgebras encode information about
the (co)bar construction, amongst other constructions on chain complexes.

Locally persentable braided monoidal closed categories are the most general
setting so far in which to develop the theory of measuring comonoids and the
associated enrichment, as is detailed in~\cite{Measuringcomonoid}. Even though
this setting is general enough to encompass the examples discussed so far, there
are
examples of interest that escape it. For example, (co)monads are
(co)monoids in a monoidal category endofunctors whose tensor product is highly
non-braided: functor composition. Operads provide another example of the same
nature. There is a monoidal category (in fact, there at least two equivalent
categories considered in the literature) equipped with a ``substitution'' tensor
product for which (co)monoids are (co)operads. If one thinks of an operad as a
sequence of spaces $A_n$ of $n$-ary operations (perhaps with an action of the
symmetric group), then the substitution of $A$ in $A$, denoted by $A\circ A$,
has $n$-ary operations that are formal substitutions $q(p_1,\dots,p_k)$ of
$m_i$-ary operators $p_i$ into $k$-ary operations $q$,
$\sum_{i=1}^km_i = n$. The multiplication $(A\circ A)_n\to A_n$ sends each
formal sustitution to the actual substitution. We encounter again the obstacle
that the substitution tensor product admits no braiding.

Duoidal categories are the answer to sidestepping the obstacles described
above. These have two interrelated monoidal structures, in the sense that one of
them is lax monoidal with respect to the other. The most vivid part of this
structure is an interchange law: a natural transformation with components
$(a\star b)\diamond(c\star d)\to (a\diamond c)\star (b\diamond d)$. The basics
of duoidal category theory were developed in~\cite{Species} along a large number
of examples of interest in algebraic combinatorics. Lax braided monoidal
categories are duoidal, with $\diamond=\star$ and the interchange law given by
braiding middle pair of objects. In this precise sense, duoidal categories
generalise braided categories.

We show the duoidal structure is rich enough to allow the definition of
measuring comonoid. Our main result is the existence of an enrichment of the
category of $\star$-monoids in the $\diamond$-monoidal category of
$\star$-comonoids, provided that the tensor product $\diamond$ is biclosed, the
duoidal category is locally presentable and the tensor product $\star$ is
accessible. The enriched homs are provided by the universal measuring comonoids.

We concentrate on the examples of graded objects and species in a sufficiently
good monoidal category. These are equipped with three monoidal structures: the
pointwise (or Hadamard) tensor product, the Cauchy tensor product (the usual
tensor product of graded objects), and the substitution tensor product. Several
duoidal structures are obtained by combining these monoidal structures. For
example, the combinations Hadamard--Cauchy and Hadamard--substitution yield
duoidal categories, considered in~\cite{Species} in the category of vector spaces. There is another
``convolution'' tensor product that together with substitution makes the
category of species a duoidal category, considered in~\cite{Commutativity} in the setting of a
cartesian closed category. The general theory developed herein produces
enriched categories from these examples. For instance, we give enrichments of the
category of symmetric operads in two different monoidal categories of symmetric
cooperads.

We give below a description of the paper's
structure. Section~\ref{sec:acti-mono-categ-1} gives an account of the
relationship between actions of monoidal categories and enriched categories, its
application to the construction of universal measuring comonoids and the
associated categories enriched in comonoids. The generalisation of measuring
comonoids to duoidal categories is introduced in
Section~\ref{sec:SweedlerTheoryduoid}. We prove in
Theorem~\ref{thm:sweedlerduoidal} that universal measuring comonoids exist and
yield an enrichment of monoids in comonoids, only with mild assumptions on the
duoidal category. Applications to categories of graded objects
$\Grd(\mathcal{V})$ and species $\Sp(\mathcal{V})$ in a sufficiently good
monoidal category $\mathcal{V}$ are in Section~\ref{sec:grad-objects-spec}. A
precise list of the enriched categories that we construct can be found in \cref{thm:allenrichments}.

\subsection*{Acknowledgements}
The authors would like to thank Martin Hyland, who provided the original
intution for the enrichment of operads in cooperads, and Marcelo Aguiar for
sharing his insights.
The first author acknowledges the partial support from {ANII} (Uruguay) and
PEDECIBA.
The second author would like to thank the General Secretariat for Research and
Technology (GSRT) and the Hellenic Foundation for Research and Innovation
(HFRI), as well as the CoACT, Macquarie University,
 where during her short stay as a Visiting Fellow, various aspects of this work were clarified.

\section{Actions of monoidal categories and enrichment}
\label{sec:acti-mono-categ-1}
\subsection{Actions of monoidal categories}
\label{sec:acti-mono-categ}

Perhaps the shortest way of describing a \emph{left lax action} of a monoidal category
$(\mathcal{V},\ot,i)$ on a category $\mathcal{C}$ is as a lax monoidal functor from
$\mathcal{V}$ to the strict monoidal category of endofunctors of
$\mathcal{C}$. The tensor product in the latter is given by composition. An
alternative description consists of a functor $\triangleright\colon
\mathcal{V}\times\mathcal{C}\to\mathcal{C}$ together with natural transformations
\begin{equation}
  \label{eq:1}
  \alpha_{x,y,a}\colon x\triangleright(y\triangleright a)\to (x\otimes
  y)\triangleright a
  \qquad
  \lambda_a\colon a\to i\triangleright a
\end{equation}
that satisfy coherence conditions,
analogous to those satisfied in a monoidal category. The one-sentence way of
describing this structure is to say that a left lax action of $\mathcal{V}$ is a
lax algebra for the pseudomonad $(\mathcal{V}\times-)$ on
$\Cat$.
Left \emph{oplax} actions of $\mathcal{V}$ are a dual version of the lax
actions, now with transformations
$\alpha_{x,y,a}\colon (x\otimes y)\triangleright a\to
x\triangleright(y\triangleright a)$ and $\lambda_a\colon i\triangleright a\to a$.

Finally, a \emph{pseudoaction,} or just \emph{action} of the monoidal category
$\mathcal{V}$, is a lax action whose transformations $a$ and $\lambda$ are invertible.
Actions of monoidal categories already appeared in~\cite{Benabou}, where they are
described as two-object bicategories.

\subsection{Parametrised ajunctions}
\label{sec:param-ajunct}
A parametrised adjunction between functors $F\colon\ca{A}\times\ca{B}\to\ca{C}$
and $G\colon\ca{B}^\op\times\ca{C}\to\ca{A}$ consists of adjunctions
$F(\mi,b)\dashv G(b,\mi)$ for each $b\in \mathcal{B}$, with unit
$a\to G(b,F(a,b))$ that is not only natural in $a\in \mathcal{A}$ but also
dinatural in $b\in\mathcal{B}$. Equivalently, the counit $F(G(b,c),b)\to c$
should be not only natural in $c$, but also dinatural in $b\in \mathcal{B}$. In
terms of homsets, we have isomorphisms $\mathcal{C}(F(a,b),c)\cong
\mathcal{A}(a,G(b,c))$ natural in all three variables $a\in\mathcal{A}$,
$b\in\mathcal{B}$, $c\in \mathcal{C}$.

There is a version of \emph{doctrinal adjunction} for parametrised adjunctions
between monoidal categories, whose details can be found
in~\cite[Prop.~3.2.3]{PhDChristina}. The natural transformations part follows analogously, considering mates under
parametrised adjunctions, see e.g. \cite[Prop.~2.11]{Multivariableadjunctions}.

\begin{lemma}\label{prop:doctrinalparameter}
 If $F\colon\ca{A}\times\ca{B}\to\ca{C}$ and $G\colon\ca{B}^\op\times\ca{C}\to\ca{A}$ are parametrised adjoints between monoidal categories, namely 
$F(\mi,b)\dashv G(b,\mi)$ for all $b\in\ca{B}$, oplax monoidal structures on $F$
bijectively correspond to lax monoidal structures on $G$. Furthermore, mateship
under the adjunctions establishes a bijection between monoidal natural
transformations $F\Rightarrow F'$ and monoidal natural transformations
$G'\Rightarrow G$.
\end{lemma}

Parametrised adjoint of actions are of central importance in what follows. The following lemma establishes a bijective correspondence between action structures of $\ca{V}$ and $\ca{V}^{\mathrm{op,rev}}$, namely the opposite category with $a\ot^\mathrm{rev}b=b\ot a$.

\begin{lemma}
  \label{l:1}
  Let $\mathcal{V}$ be a monoidal category, $\mathcal{A}$ a category, and a
  parametrised adjunction between functors
  $\triangleright\colon\mathcal{V}\times\mathcal{A}\to\mathcal{A}$ and
  $\pitchfork\colon
  \mathcal{V}^{\mathrm{op}}\times\mathcal{A}\to\mathcal{A}$ via $(x\triangleright\mi)\dashv(x\pitchfork\mi)$. There exists a
  bijection between:
  \begin{enumerate*}[label=(\roman*)]
  \item
    associativity and unit transformations that make
    $\triangleright$ a lax action of $\mathcal{V}$;
  \item associativity and unit transformations that make $\pitchfork$ a oplax
    action of $\mathcal{V}^{\mathrm{op,rev}}$.
\end{enumerate*}
Furthermore, $\triangleright$ is an action if and only if the corresponding
$\pitchfork$ is an action.
\end{lemma}

\begin{proof}
  The proof is an application of
  general fact about mates: given functors $F\dashv F^r$ and $G\dashv G^r$,
  mateship induces a bijection between natural transformations $F\Rightarrow G$
  and $G^r\Rightarrow F^r$. Furthermore, one is invertible if and only if the
  other is so. 

  Applying the above to the adjunctions
  $(x\triangleright\mi)\dashv(x\pitchfork\mi)$,
  we have a bijection between natural transformations
  $\alpha_{x,y,a}\colon x\triangleright(y\triangleright {a})\Rightarrow (x\otimes
  y)\triangleright{a}$ and natural transformations
  $\bar\alpha_{a,x,y}\colon (x\otimes y)\pitchfork {a}\Rightarrow
  y\pitchfork(x\pitchfork a)$,
  natural $a\in \mathcal{A}$, $x,y\in \mathcal{V}$. 

  Similarly, there is a
  bijection between transformations
  $\lambda_a\colon a \Rightarrow (i\triangleright{a})$ and
  $\bar\lambda_a\colon i\pitchfork a\Rightarrow a$. The verification that
  that $\alpha,\lambda$ satisfy the axioms of a lax action if
  and only if $\bar\alpha,\bar\lambda$ satisfy those of a oplax action is
  straightforward. Finally, $\alpha$ is invertible if and only if $\bar \alpha$
  is so, and $\lambda$ is invertible if and only if $\bar\lambda$ is so.
\end{proof}

\begin{rmk}
  \label{rmk:2}
  If the monoidal category $\mathcal{V}$ of Lemma~\ref{l:1} has a braiding
  $\gamma$, then $\pitchfork$ can be made into an action of
  $\mathcal{V}^{\mathrm{op}}$ using the strong monoidal structure on the
  identity functor $\mathcal{V}\to\mathcal{V}^{\mathrm{rev}}$ induced by
  $\gamma$. Expicitly, the associativity of the action of $\mathcal{V}$ has
  components
  \begin{equation}
    \label{eq:11}
    x\pitchfork (y\pitchfork a)\cong (y\otimes x)\pitchfork a
    \xrightarrow{\gamma_{x,y}\pitchfork 1}
    (x\otimes y)\pitchfork a
  \end{equation}
  where the unlabelled isomorphism is the associativity of the action of
  $\mathcal{V}^{\mathrm{rev}}$ provided by the lemma.
\end{rmk}

\begin{df}
  \label{df:1}
  A left action $\triangleright\colon\mathcal{V}\times\mathcal{A}\to\mathcal{A}$ of
  the monoidal category $\mathcal{V}$ is said to be \emph{closed} if
  $\triangleright$ has a parametrised right adjoint
  $\mathcal{A}^{\mathrm{op}}\times\mathcal{A}\to \mathcal{V}$, which we call the
  \emph{enriched hom}. 
\end{df}
If, in the definition above, we denote the parametrised right adjoint by $P$,
then, we have natural bijections $\mathcal{A}(x\triangleright
a,b)\cong\mathcal{V}(x,P(a,b))$, for $x\in\mathcal{V}$, $a,b\in\mathcal{A}$.

For example, the regular left 
action of a monoidal category $\mathcal{V}$ on itself,
via the tensor product, is a closed left  
action precisely when
$\mathcal{V}$ is a monoidal left 
closed category.

\subsection{Enrichment, tensor and cotensor products}
\label{sec:tens-enrichm-acti}

The widespread use of enriched categories in modern Mathematics alleviates the
need of recalling all the details of their definition; for a detailed exposition, see \cite{Kelly}. The monoidal category
used as base for the enrichment is most of the time a symmetric monoidal closed
category; think of the category of abelian groups, or the category of simplicial
sets. In some of the examples we shall develop in later secions, the base of
enrichment is not even braided. In this context the theory of enriched
categories differs very little from the theory of categories enriched in a
bicategory; see for example~\cite{Varthrenr}. We shall recall some the parts of
the theory that makes it distinct from the case of symmetric monoidal
categories, taking the opportunity to set out our conventions, the absence of
which can lead to confusion. 

Let $(\mathcal{V},\ot,i)$ be a monoidal category. The first convention we set is the
order of composition. If $\enr{A}$ is a $\mathcal{V}$-category, its composition
is given by morphisms $\enr{A}(b,c)\otimes\enr{A}(a,c)\to\enr{A}(a,c)$.

The underlying category of a $\mathcal{V}$-category $\enr{A}$, usually denoted
by $\enr{A}_0$, has the same objects as $\enr{A}$ and homsets
$\enr{A}_0(a,b)=\mathcal{V}(i,\enr{A}(a,b))$. We shall sometimes write
$\mathcal{A}$ for $\enr{A}_0$, when no confusion is likely.

Now assume that the monoidal category $\mathcal{V}$ is left closed, so we have
natural isomorphisms $\mathcal{V}(x\otimes y,z)\cong\mathcal{V}(x,\lhom{y}{z})$. Given
objects $x\in\mathcal{V}$ and $a\in \enr{A}$, a \emph{tensor} of $a$ by $x$ is
an object $x\triangleright a$ of $\enr{A}$ together with a morphism $\eta\colon
x\to\enr{A}(a,x\triangleright  a)$ that satisfies: the morphism
$\enr{A}(x\triangleright a,b)\to\lhom{x}{\enr{A}(a,b)}$ corresponding to the composition
\begin{equation}
  \label{eq:2}
  \enr{A}(x\triangleright  a,b)\otimes x\xrightarrow{1\otimes\eta}
  \enr{A}(x\triangleright  a,b)\otimes\enr{A}(a,x\triangleright  a)\xrightarrow{\mathrm{comp}}
  \enr{A}(a,b)
\end{equation}
is an isomorphism, for all $b\in \enr{A}$.  Note that the isomorphism
$\enr{A}(x\triangleright  a,b)\cong\lhom{x}{\enr{A}(a,b)}$ is $\mathcal{V}$-natural in $b$. If
$\mathcal{A}$ is the underlying category of $\enr{A}$, there is an induced
bijection $\mathcal{A}(x\triangleright  a,b)\cong\mathcal{V}(x,\enr{A}(a,b))$, that, being
natural in $b\in\mathcal{A}$, gives rise to an ordinary functor
$\triangleright  \colon \mathcal{V}\times\mathcal{A}\to\mathcal{A}$.

Tensor products have a dual notion, called cotensor products. The monoidal
category $\mathcal{V}$ needs now to be right closed, so we have natural
isomorphisms $\mathcal{V}(x\otimes y,z)\cong\mathcal{V}(y,\rhom{x}{z})$. Given
objects $x\in \mathcal{V}$ and $b$ in the $\mathcal{V}$-category $\enr{A}$, a
\emph{cotensor} of $b$ by $x$ is an object $x\pitchfork b\in \enr{A}$ together with a
morphism $\eta\colon x\to \enr{A}(x\pitchfork b,b)$ in $\mathcal{V}$ that
satisfies: the morphism $\enr{A}(a,x\pitchfork b)\to\rhom{x}{\enr{A}(a,b)}$
corresponding to the composition
\begin{equation}
  \label{eq:3}
  x\otimes\enr{A}(a,x\pitchfork b)
  \xrightarrow{\eta\otimes 1}
  \enr{A}(x\pitchfork b,b)\otimes \enr{A}(a,x\pitchfork b)
  \xrightarrow{\mathrm{comp}}
  \enr{A}(a,b)
\end{equation}
is an isomorphism, for all $a\in \enr{A}$.
Analogously to the previous paragraph, there is a natural bijection
$\mathcal{A}(a,x\pitchfork b)\cong\mathcal{V}(x,\enr{A}(a,b))$ and an ordinary functor
$\pitchfork\colon
\mathcal{V}^{\mathrm{op}}\times\mathcal{A}\to\mathcal{A}$.

In the case when the monoidal category $\mathcal{V}$ is biclosed, there is a
parametrised adjunction between $\triangleright$ and $\pitchfork$, presented by the first natural
bijection of
\begin{equation}\label{eq:tens-cotens-hom}
\mathcal{A}(x\triangleright a,b)\cong\mathcal{A}(a,x\pitchfork b)\cong\mathcal{V}(x,\enr{A}(a,b))
\end{equation}
capturing the sometimes called tensor-cotensor-hom adjunction.

In the absence of a braiding for the monoidal category $\mathcal{V}$, there is
is no obvious way of defining the opposite $\mathcal{V}$-category $\enr{A}$. We
are able, though, to define a $\mathcal{V}^{\mathrm{rev}}$-category that we
write $\enr{A}^{\mathrm{rev}}$. It has the same objects as $\enr{A}$, and homs
$\enr{A}^{\mathrm{rev}}(a,b)=\enr{A}(b,a)$. The verification that the
composition and identities of $\enr{A}$ make $\enr{A}^{\mathrm{rev}}$ a
$\mathcal{V}^{\mathrm{rev}}$-category is straightforward. Observe that the
underlying category of $\enr{A}^{\mathrm{rev}}$ is the opposite of the
underlying category of $\enr{A}$; in symbols
$(\enr{A}^{\mathrm{rev}})_0=(\enr{A}_0)^{\mathrm{op}}$.

Left internal homs for $\mathcal{V}$ are right internal homs for
$\mathcal{V}^{\mathrm{rev}}$. This, and an inspection of the definitions of tensor and
cotensor products given above, will tell us that a cotensor product
$x\pitchfork a$ in the $\mathcal{V}$-category $\enr{A}$ is the same as a tensor
product $x\triangleright a$ in the $\mathcal{V}^{\mathrm{rev}}$-category
$\enr{A}^{\mathrm{rev}}$.

\subsection{Actions and enrichment}
\label{sec:actions-enrichment}

There is a close relationship between actions and enriched
categories, explored in \cite{enrthrvar} in the context of bicategories, and
\cite[\S 6]{AnoteonActions} in the context of monoidal categories. (The
later article calls right closed what we call left closed.) We state part of the said
relationship in the following theorem.
\begin{thm}
  \label{thm:1}
  Let
  $\triangleright\colon\mathcal{V}\times\mathcal{A}\to\mathcal{A}$ be a closed left
  action of the monoidal category
  $\mathcal{V}$ on the category $\mathcal{A}$, with enriched hom $P\colon
  \mathcal{A}^{\mathrm{op}}\times\mathcal{A}\to\mathcal{V}$. Then, there exists
  a $\mathcal{V}$-category $\enr{A}$ with
  \begin{enumerate*}[label=(\roman*)]
  \item
    the same objects as $\mathcal{A}$;
  \item homs $\enr{A}(a,b)=P(a,b)$;
  \item underlying category $\mathcal{A}$.
  \end{enumerate*}
  Moreover, if $\mathcal{V}$ is left closed, then $x\triangleright a$ is the
  tensor product of $a\in\enr{A}$  by $x\in \mathcal{V}$.
\end{thm}
The composition of the $\mathcal{V}$-category $\mathsf{A}$ has components
$P(b,c)\otimes P(a,b)\to P(a,c)$ that correspond under the parametrised
adjunction to the following composite of action associativity and counits.
\begin{equation}
  (P(b,c)\otimes P(a,c))\triangleright a
  \cong
  P(b,c)\triangleright(P(a,c)\triangleright a)
  \to
  P(a,c)\triangleright a
  \to
  c
  \label{eq:19}
\end{equation}

Part of the following result was established in \cite[Lem.~2.1]{AnoteonActions}, however in the presence of symmetry which is here avoided.
\begin{prop}
  \label{prop:1}
  Let $\triangleright\colon\mathcal{V}\times\mathcal{A}\to\mathcal{A}$ be a
  closed action of the right closed monoidal category $\mathcal{V}$, and $\enr{A}$
  the associated $\mathcal{V}$-category, as in Theorem~\ref{thm:1}. The
  following statements are equivalent.
  \begin{enumerate}
  \item \label{item:1}
    For each $x\in\mathcal{V}$ the endofunctor $(x\triangleright-)$ of
    $\mathcal{A}$ has a right adjoint.
  \item \label{item:2}
    $\enr{A}$ has cotensor products.
  \end{enumerate}
\end{prop}
\begin{proof}
  Recall that we have an adjunction
  $\mathcal{A}(x\triangleright a,b)\cong\mathcal{V}(x,\enr{A}(a,b))$. Assuming
  (\ref{item:1}) we shall prove (\ref{item:2}). Let us denote the right adjoint of
  $x\triangleright\mi$ by $x\pitchfork\mi$. For objects $x\in\mathcal{V}$ and
  $b\in\mathcal{A}$, define a morphism
  $\eta\colon x \to\enr{A}(x\pitchfork b,b)$ corresponding to the component
  $x\triangleright (x\pitchfork b)\to b$ of the unit.
  We shall show that the morphisms
  $\theta\colon \enr{A}(a,x\pitchfork b)\to\rhom{x}{\enr{A}(a,b)}$ induced by $\eta$ are
  isomorphisms, for all $a\in\mathcal{A}$. The result of applying the
  representable functor $\mathcal{V}(y,-)$ to $\theta$ can be written as the
  following composition of isomorphisms.
  This for all $y$, hence $\theta$ is an isomorphism.
  \begin{multline}
    \label{eq:6}
    \mathcal{V}(y,\enr{A}(a,x\pitchfork b))
    \cong
    \mathcal{A}(y\triangleright a,x\pitchfork b)
    \cong
    \mathcal{A}(x\triangleright (y\triangleright a),b)
    \cong
    \mathcal{A}((x\otimes y)\triangleright a,b)
    \cong
    \mathcal{V}(x\otimes y,\enr{A}(a,b))
    \\
    \cong
    \mathcal{V}(y,\rhom{x}{\enr{A}(a,b)})
  \end{multline}

  We now assume (\ref{item:2}) and shall prove (\ref{item:1}). If the cotensor
  product $x\pitchfork b$ exists, then, we have isomorphisms, natural in
  $a\in\mathcal{A}$
  \begin{equation}
    \label{eq:7}
    \mathcal{A}(x\triangleright a,b)\cong\mathcal{V}(x,\enr{A}(a,b))
    \cong
    \mathcal{V}(i,\enr{A}(a,x\pitchfork b))\cong
    \mathcal{A}(i\triangleright a,x\pitchfork b)
    \cong
    \mathcal{A}(a,x\pitchfork b).
  \end{equation}
\end{proof}

It will be usefull to set out in the following proposition and corollary
the type of action and enrichment we will encounter more often, as to avoid
having to make deductions about it in each instance.

\begin{prop}
  \label{prop:4}
  Let $\mathcal{V}$ be a monoidal category, and
  $\pitchfork\colon\mathcal{V}^{\mathrm{op}}\times\mathcal{A}\to\mathcal{A}$
  be a action of $\mathcal{V}^{\mathrm{op,rev}}$ with a parametrised adjoint
  $\mathcal{A}(a, x\pitchfork b)\cong\mathcal{V}(x,P(a,b))$.
  \begin{enumerate}[wide]
  \item \label{item:5} There is a $\mathcal{V}$-category $\mathsf{B}$ with:
    \begin{enumerate*}[label=(\roman*)]
    \item the same objects as $\mathcal{A}$;
    \item internal homs $\mathsf{B}(a,b)=P(a,b)$;
    \item underlying category $\mathcal{A}$;
    \end{enumerate*}
its composition $P(b,c)\otimes P(a,b)\to P(a,c)$ corresponds under the parametrised ajunction to the composite morphism below.
    \begin{equation}
      a\to
      P(a,b)\pitchfork b
      \to
      P(a,b)\pitchfork (P(b,c)\pitchfork c)
      \cong
      (P(b,c)\otimes P(a,b))\pitchfork c
      \label{eq:21}
    \end{equation}
  \item \label{item:6}  Assume that $\mathcal{V}$ is right closed. 
    Then $\mathsf{B}$ is cotensored,
    and  $x\pitchfork a$ is the cotensor product of $a\in\mathsf{B}$ by
    $x\in\mathcal{V}$. 
  \item \label{item:7} Assume that $\mathcal{V}$ is left closed. Then $\mathsf{B}$ is tensored if and only if each
    functor $x\pitchfork\mi$ has a left adjoint $x\triangleright\mi$, in which
    case $x\triangleright a$ is the tensor product of $a\in\mathsf{B}$ by
    $x\in\mathcal{V}$.
  \end{enumerate}
\end{prop}
\begin{proof}
  (\ref{item:5}) The parametrised adjoint $P$ provides an enriched hom
  $Q\colon\mathcal{A}\times \mathcal{A}^{\mathrm{op}}\to\mathcal{V}$ for the
  action $\pitchfork^{\mathrm{op}}$ of $\mathcal{V}^{\mathrm{rev}}$ on
  $\mathcal{A}^\mathrm{op}$. Here $Q(a,b)=P(b,a)$. There is, by
  Theorem~\ref{thm:1}, a $\mathcal{V}^{\mathrm{rev}}$-category $\mathsf{A}$ with the same
  objects as $\mathcal{A}$, with underlying category isomorphic to
  $\mathcal{A}^{\mathrm{op}}$, and with internal homs
  $\mathsf{A}(a,b)=Q(a,b)$. The composition of $\mathsf{A}$ is described after
  Theorem~\ref{thm:1}. It is the morphism
  \begin{equation}
    P(b,a)\otimes P(c,b)=
    Q(b,c)\otimes^{\mathrm{rev}} Q(a,b) \to Q(a,c)
    =P(c,a)
    \label{eq:12}
  \end{equation}
  that
  correponds to the composition in $\mathcal{A}$ of adjoint units and the
  associativity of the action, as shown.
  \begin{equation}
    \label{eq:22}
    c\to Q(b,c)\pitchfork b
    \to Q(b,c)\pitchfork(Q(a,b)\pitchfork a)
    \cong (Q(a,b)\otimes Q(b,c))\pitchfork a
  \end{equation}
  Let $\mathsf{B}$ be the $\mathcal{V}$-category with the same objects as
  $\mathsf{A}$ and homs $\mathsf{B}(a,b)=\mathsf{A}(b,a)=P(a,b)$, with
  composition given by~\eqref{eq:12} and the same identities as $\mathsf{A}$.
  This composition clearly has transpose~\eqref{eq:21}, as required.
  Finally, the underlying category of $\mathsf{B}$ is $\mathcal{A}$, since that
  of $\mathsf{A}$ is $\mathcal{A}^{\mathrm{op}}$.

  (\ref{item:6}) Let's assume that $\mathcal{V}$ is right closed, which is to
  say that $\mathcal{V}^{\mathrm{rev}}$ is left closed. We know that the
  $\mathcal{V}^{\mathrm{rev}}$-category $\mathsf{A}$ is tensored by
  Theorem~\ref{thm:1}, with tensors $x\pitchfork a$. This is equivalent to
  saying that $\mathsf{B}$ is cotensored with cotensors $x\pitchfork a$.

  (\ref{item:7}) When $\mathcal{V}$ is left closed, $\mathcal{V}^{\mathrm{rev}}$
  is right closed, so $\mathsf{A}$ is cotensored if and only if each
  $(x\pitchfork\mi)^{\mathrm{op}}$ has a right adjoint, in which case the
  cotensor is provided by this adjoint, by
  Proposition~\ref{prop:1}. Translating this in terms of $\mathsf{B}$, we obtain
  the statement.
\end{proof}

\begin{cor}
  \label{cor:5}
  Consider a biclosed monoidal closed category $\mathcal{V}$ and a
  triplete of functors
  $\triangleright\colon\mathcal{V}\times\mathcal{A}\to\mathcal{A}$,
  $P\colon\mathcal{A}^{\mathrm{op}}\times\mathcal{A}\to\mathcal{V}$ and
  $\pitchfork\colon\mathcal{V}^{\mathrm{op}}\times\mathcal{A}\to\mathcal{A}$,
  together with parametrised adjunctions
  \begin{equation}
    \label{eq:9}
    \mathcal{A}(x\triangleright a,b)\cong\mathcal{V}(x,P(a,b))\cong
    \mathcal{A}(a,x\pitchfork b).
  \end{equation}
  Assume that $\triangleright$ is an action of $\mathcal{V}$ on $\mathcal{A}$,
  or, equivalently, that $\pitchfork$ is an action of
  $\mathcal{V}^{\mathrm{op,rev}}$ (\cref{l:1}). Then, the
  $\mathcal{V}$-category arising from $\triangleright$ (Theorem~\ref{thm:1}) is
  the same as the $\mathcal{V}$-category arising from $\pitchfork$
  (Proposition~\ref{prop:4}).  This $\mathcal{V}$-category has tensor products
  provided by $x\triangleright a$ and cotensor products provided by
  $x\pitchfork a$.
\end{cor}
\begin{proof}
  We have two $\mathcal{V}$-categories, $\mathsf{A}$ from Theorem~\ref{thm:1}
  and $\mathsf{B}$ from Proposition~\ref{prop:4}. These have the same objects as well as homs,
  provided by $P(a,b)$. One has to verify that the composition and identities
  coincide.

  Given one element $f$ of the homset on the left, and another element $g$ of the homset
  on the right, they have the same image in the middle homset precisely if $g$
  is the transpose of $f$ under the adjunction $(P(b,c)\otimes
  P(a,b)\triangleright\mi)\dashv (P(b,c)\otimes
  P(a,b)\pitchfork\mi)$.
  \begin{equation}
    \label{eq:4}
    \mathcal{A}\bigl((P(b,c)\otimes P(a,b))\triangleright a,P(a,c)\bigr)
    \cong
    \mathcal{V}\bigl(P(b,c)\otimes P(a,b),P(a,c)\bigr)
    \cong
    \mathcal{A}\bigl(a,(P(b,c)\otimes P(a,b))\pitchfork c\bigr)
  \end{equation}
  We know that the morphism \eqref{eq:19}, an element of the homset on the left,
  corresponds to the composition of $\mathsf{A}$ in the middle homset.
  Similarly, the morphism~\eqref{eq:21}, an element of the homset on the right,
  corresponds to the composition of $\mathsf{B}$ in the middle homset. Moreover,
  \eqref{eq:21} is the transpose of \eqref{eq:19}, as it is readily shown. This
  completes the proof that the two compositions coincide. We leave to the reader
  the proof that the identities coincide.
\end{proof}

\subsection{Sweedler theory for monoidal categories}\label{sec:STmoncats}

We now recall the basic results regarding the so-called \emph{Sweedler Theory} for a braided monoidal closed category. Details of what follows can 
be found in various sources \cite{mine,AnelJoyal,PhDChristina,Measuringcomonoid} and generalized for double categories in \cite{VCocats}. The name 
comes from Sweedler's \cite{Sweedler} where \emph{measuring coalgebras}, which ultimately induce a $\Coalg$-enrichment of algebras, were introduced. 

In any monoidal category $(\ca{V},\ot,I)$ we can form categories of monoids $\Mon(\ca{V})$ and comonoids $\Comon(\ca{V})$ of objects with a (co)associative and (co)unital (co)mu\-lti\-pli\-ca\-tion. If $\ca{V}$ is braided, both categories inherit the monoidal structure as in
\begin{equation}\label{eq:monmon}
 a\ot b\ot a\ot b\cong a\ot a\ot b\ot b\xrightarrow{\mu\ot\mu}a\ot b
\end{equation}
Moreover, if $\ca{V}$ is braided (left) monoidal closed, its internal hom has a natural lax monoidal structure $[a,b]\ot[a',b']\to[a\ot a',b\ot b']$ 
which corresponds to
\begin{displaymath}
 [a,b]\ot[a',b']\ot a\ot a'\cong [a,b]\ot a\ot [a',b']\ot a'\to b\ot b'
\end{displaymath}
using the braiding and evaluation maps. As a result, $[\mi,\mi]\colon \ca{V}^\op\times\ca{V}\to\ca{V}$ lifts between the categories of monoids
\begin{equation}\label{eq:actioncomon}
 [\mi,\mi]\colon\Comon(\ca{V})^\op\times\Mon(\ca{V})\to\Mon(\ca{V})
\end{equation}
which establishes the monoid structure of any $[c,b]$ for $c$ a comonoid and $b$ a monoid via convolution.

Moreover, it can be verified that $[\mi,\mi]$ is an \emph{action} via the standard $[a\ot 
b,c]\cong[a,[b,c]]$ 
and $[I,a]\cong a$, see \cite[Lem.~4.3.2]{PhDChristina}.

\begin{prop}\label{prop:homaction}
If $\ca{V}$ is (left) monoidal closed, $[\mi,\mi]$ is an action of $\ca{V}^\op$ on $\ca{V}$; when $\ca{V}$ is moreover braided, its induced functor 
between (co)monoids is an 
action of $\Comon(\ca{V})^\op$ on $\Mon(\ca{V})$.
\end{prop}

As a result, for a symmetric monoidal closed category where $\mi\ot a\dashv [a,\mi]$ but also $[\mi,b]^\op\dashv[\mi,b]$, 
\cref{cor:5} applies and provides the well-known result that $\ca{V}$ is tensored (via $\otimes$) and cotensored (via 
$[\mi,\mi]$) enriched in itself, with $\Hom_\ca{V}=[\mi,\mi]$.

In the case of (co)monoids, the existence of a parametrised adjoint to the action $[\mi,\mi]$ would also induce a cotensored and possibly tensored enrichment according to \cref{sec:actions-enrichment}, and that is precisely what the existence of the `Sweedler hom' functor $P\colon\Mon(\ca{V})^\op\times\Mon(\ca{V})\to\Comon(\ca{V})$ and the `Sweedler product' functor $\mi\triangleright\mi\colon\Comon(\ca{V})\times\Mon(\ca{V})\to\Mon(\ca{V})$ establishes below.

For the existence of adjoints for \cref{eq:actioncomon}, we restrict to the class of locally presentable categories. 
By working in this sufficient context, we can use the fact that locally presentable categories are cocomplete and contain a small dense 
subcategory (of presentable objects) so that every cocontinuous functor with such a domain has a right adjoint -- as a variation of the Special 
Adjoint Functor Theorem, see \cite[Thm.~5.33]{Kelly}. 

Although the original references for the following result by Porst assume symmetric monoidal closed structure on $\ca{V}$ for simplicity, our 
subsequent examples require that the results should be properly stated in a
weaker setting, see also \cite[45]{PhDChristina}, \cite[Prop.~2.9]{Measuringcomonoid}.

\begin{prop}\label{prop:MonComonproperties}
Suppose that $\ca{V}$ is a locally presentable category, equipped with a
monoidal structure whose tensor product is accessible.
Then the categories of monoids $\Mon(\ca{V})$ and comonoids $\Comon(\ca{V})$
are locally presentable, the former is monadic and latter comonadic over
$\ca{V}$.
Furthermore, if $\mathcal{V}$ is braided, then $\Comon(\ca{V})$ is monoidal
biclosed.
\end{prop}

\begin{thm}\label{thm:measucoalg}\cite[Thm.~4.1]{Measuringcomonoid},\cite[105]{PhDChristina}
 For a braided monoidal closed and locally presentable category $\ca{V}$, each functor $[\mi,b]^\op\colon\Comon(\ca{V})\to\Mon(\ca{V})^\op$ has a 
right adjoint $P(\mi,b)$ and also each functor $[c,\mi]\colon\Mon(\ca{V})\to\Mon(\ca{V})$ has a left adjoint $c\triangleright\mi$, via natural 
bijections
\begin{displaymath}
 \Mon(\ca{V})(c\triangleright a,b)\cong\Comon(\ca{V})(c,P(a,b))\cong\Mon(\ca{V})(a,[c,b])
\end{displaymath} 
\end{thm}

Now \cref{cor:5} applies, in a slightly more straightforward way than the 
original proofs in the literature.

\begin{thm}\label{thm:STmoncats}\cite[Thm.~6.1.4]{PhDChristina},\cite[Thm.~5.2]{Measuringcomonoid}
 The category of monoids $\Mon(\ca{V})$ is tensored and cotensored enriched in
 the category of comonoids $\Comon(\ca{V})$, for any locally presentable, braided monoidal closed $\ca{V}$.
\end{thm}

In fact, in \cite{Measuringcomonoid} it is shown that this enrichment is \emph{monoidal}.
The three functors of two variables that express the tensor, the enriched hom and the cotensor -- compare to \cref{eq:tens-cotens-hom} -- are
\begin{gather}
 [\mi,\mi]\colon\Comon(\ca{V})^\op\times\Mon(\ca{V})\longrightarrow\Mon(\ca{V})\nonumber \\
 P(\mi,\mi)\colon\Mon(\ca{V})^\op\times\Mon(\ca{V})\longrightarrow\Comon(\ca{V})\nonumber \\
 \mi\triangleright\mi\colon\Comon(\ca{V})\times\Mon(\ca{V})\longrightarrow\Mon(\ca{V})\label{eq:monoidalsweedler}
\end{gather}
Along with the inherited tensor product functor on monoids and comonoids, as well as the internal hom of comonoids 
$\HOM\colon\Comon(\ca{V})^\op\times\Comon(\ca{V})\to\Comon(\ca{V})$ of comonoids, these six functors were collectively called \emph{Sweedler Theory} 
in the context of unpointed dg-algebras and dg-coalgebras in \cite{AnelJoyal}; we adopt this convention in this more general setting. The functor $\triangleright$ is called the \emph{Sweedler product} and the enriching hom-functor $P$ is called \emph{Sweedler hom}.

In fact, the \emph{universal measuring coalgebra} $P(a,b)$ of two $k$-algebras $a,b$ was introduced by Sweedler in \cite{Sweedler}, where it was explicitly constructed (as a sum of subcoalgebras of the cofree coalgebra on $\Hom_k(a,b)$) and verified that it satisfies the universal property of the above adjunction in $\ca{V}=\Vect_k$, see Theorem 7.0.4 therein. Broadly speaking, $P(a,b)$ is to be thought of as generalised monoid maps from $a$ to $b$; in fact, the monoid maps are precisely the group-like elements of $P(a,b)$,
namely those for which $\delta_{P(a,b)}(f)=f\ot f$ and $\epsilon_{P(a,b)}(f)=1$. 
Moreover, the special case functor $P(\mi,I)=(\mi)^\circ$ is the \emph{Sweedler dual} functor for which we have that in $\Vect_k$, algebra maps $a\to c^*$ to the linear dual of a coalgebra bijectively correspond to algebra maps $c\to a^\circ$ where $a^\circ$ is the subspace of $a^*$ of linear functions whose kernel contains a cofinite ideal.

For more details and information about such structures, the reader should refer to references such as \cite{MeasuringCoalgebras,Measuringcomonoid,Sweedlerdual,AnelJoyal,Onbimeasurings}.

\section{Sweedler theory for duoidal categories}\label{sec:SweedlerTheoryduoid}

In this section, we extend the central \cref{thm:STmoncats} to the context of
duoidal categories, or 2-monoidal categories, introduced
in~\cites{Species,MR2506256}. Our results are an extension of those valid for
braided monoidal categories since the latter can be regarded as special duoidal
categories. The extension is meaningful due to the number of different examples
arising from duoidal categoris that we shall be able to cover.

\subsection{Duoidal categories}
\label{sec:duoidal-categories}

Abstractly, a duoidal category is a pseudomonoid in the cartesian monoidal 2-category $\nc{MonCat}_\ell$ of monoidal categories, lax monoidal functors and monoidal natural transformations -- or equivalently a pseudomonoid in $\nc{MonCat}_c$ with oplax monoidal functors.

\begin{df}\label{def:duoidal}
  A \emph{duoidal} category $(\ca{V},\diamond,i,\star,j)$ is a category with two
  monoidal structures $(\ca{V},\diamond,I)$ and $(\ca{V},\star,j)$ along with a
  natural transformation
 \begin{equation}\label{eq:interchange}
  \zeta_{a,b,c,d}\colon (a\star b)\diamond (c\star d)\to(a\diamond c)\star(b\diamond d)
 \end{equation}
called the \emph{interchange law} and three morphisms
\begin{equation}\label{eq:deltamuiota}
 \delta_i\colon i\to i\star i,\quad\mu_j\colon j\diamond j\to j,\quad\iota\colon i\to j
\end{equation}
subject to axioms that express that the functors $\star,j$ are lax
$\diamond$-monoidal and the associativity and unit constraints of the monoidal
structure $(\star,j)$ are $\diamond$-monoidal transformations. These axioms
express equally well that and $\diamond,i$ are oplax $\star$-monoidal and the
associativity and unit constraints of $(\diamond,i)$ are $\star$-monoidal
natural transformations.
\end{df}

There is a wealth of duoidal category examples in the literature, \cite{Species}
being a good source. We do not intend to give an exhaustive list of references,
only to mention that other examples can be found in \cite{MR2506256}, related to
factorisation systems, and in~\cite{Commutativity}, related to alrebraic
theories and operads.

Under certain conditions, there is a procedure to construct a duoidal functor category out
of a pair of monoidal categories. We shall use instances the following
proposition in the subsequent sections. The proof can be found in
Appendix~\ref{sec:duoid-struct-categ}, where the subject is treated with an
extra degree of generality.

\begin{prop}\label{prop:duoidalpresheaves}
  Let $(\mathcal{A},\otimes, k)$ be a small monoidal category and
  $(\mathcal{B},\square,i,\gamma)$ a cocomplete braided monoidal category. Then,
  $[\mathcal{A},\mathcal{B}]$ has a duoidal structure $(\conv,J,\square,i)$
  where $(\square,i)$ is the pointwise monoidal structure and $(\conv,J)$ is
  the convolution of $(\otimes,k)$ with $(\square,i)$.
\end{prop}

For the reader's benefit we recall the (Day) convolution $(\conv,J)$ of $(\otimes,k)$
with $(\square,i)$, originating from \cite{DayConv}. For functors
$F,G\colon\mathcal{A}\to\mathcal{B}$, their convolution is the left Kan
extension of $F\square G$ along $\otimes$, whereas the unit $J$ is the left Kan
extension of $i$ along $k$.
\begin{equation}
  \label{eq:5}
  \begin{tikzcd}
    \mathcal{A}\times\mathcal{A}
    \ar[d,"\otimes"']
    \ar[r,"F\times G"{name=x}]
    \ar[d,shift left = 30pt,Rightarrow, shorten <=4pt,shorten >=4pt]
    &
    \mathcal{B}\times\mathcal{B}
    \ar[d,"\square"]
    \\
    \mathcal{A}
    \ar[r,"F\conv G"']
    &
    \mathcal{B},
  \end{tikzcd}
  \qquad
  \begin{tikzcd}
    \mathsf{1}
    \ar[d,"k"']
    \ar[r,""{name=x},equal]
    \ar[d,shift left = 21pt,Rightarrow,shorten <=4pt,shorten >=4pt]
    &
    \mathsf{1}
    \ar[d,"i"]
    \\
    \mathcal{A}
    \ar[r,"J"']
    &
    \mathcal{B}
  \end{tikzcd}
\end{equation}
The existence of the associativity and unit constrains depend upon the fact that
$\mathcal{B}$ is a cocomplete monoidal category, i.e., it is cocomplete and its
tensor product is cocontinuous in each variable.

Writting these Kan extensions as colimits, one has the familiar formulas in
terms of coends and copowers.
\begin{equation}
  \label{eq:dayconv}
  (F\conv G)(a)\cong
  \int^{b,c\in\mathcal{A}}\mathcal{A}(b\otimes c,a)\cdot
  F(b)\square G(c)
  \quad
  J(a)\cong \mathcal{A}(k,a)\cdot i
\end{equation}
The convolution
monoidal structure is left or right closed when $(\mathcal{B},\square,i)$ is
left or right closed and $\mathcal{B}$ is complete. The left or right intenal
homs admit the  expressions below.
\begin{equation}
  \label{eq:dayhom}
  [F,G]_\ell(a)\cong\int_{b\in\mathcal{A}}[F(b),G(a\otimes b)]_\ell
  \qquad
  [F,G]_r(a)\cong\int_{b\in\mathcal{A}}[F(b),G(b\otimes a)]_r
\end{equation}
We take the opportunity of our digression into convolution monoidal structures
to mention that a brading $\gamma^A$ on $\mathcal{A}$ together with a brading
$\mathcal{B}^B$ on
$\mathcal{B}$ induce a braiding for the convolution tensor product. The
components of this brading are induced by the universal property of coends and
the natural transformation
\begin{equation}
  \label{eq:25}
  \mathcal{A}(b\otimes c,a)\cdot F(b)\square G(c)
  \xrightarrow{\mathcal{A}((\gamma^A_{b,c})^{-1},a)\cdot \gamma^B_{F(b),G(c)}}
  \mathcal{A}(c\otimes b,a)\cdot G(c)\square F(b).
\end{equation}

Given a duoidal category as in Definition~\ref{def:duoidal}, the monoidal
structure $(\star,j)$ lifts to the category of
$\diamond$-monoids $\Mon(\ca{V},\diamond,i)$, which we abbreviate
$\Mon_{\diamond}(\ca{V})$. For instance, if $a$ and $b$ are
$\diamond$-monoids, then $a\star b$ has multiplication and unit
\begin{gather}
 (a\star b)\diamond(a\star b)\xrightarrow{\zeta}(a\diamond a)\star(b\diamond b)\xrightarrow{\mu\star\mu}a\star b \label{eq:starMon}\\
i\xrightarrow{\delta_i}i\star i\xrightarrow{\eta\star\eta}a\star b \nonumber
\end{gather}
where $\mu$ and $\eta$ depict the respective multiplications and units; compare to the earlier \cref{eq:monmon}. Dually,
the monoidal structure $(\diamond, i)$ lifts to the category
$\Comon_\star(\ca{V})$ of $\star$-comonoids in $\ca{V}$.

If the duoidal category $\ca{V}$ of Definition~\ref{def:duoidal} is
$\star$-braided (symmetric) in the sense of~\cite[\S 6.15]{Species}, then its
braiding (symmetry) lifts to $\Mon_\diamond(\ca{V})$. Dually, if $\ca{V}$ is
$\diamond$-braided (symmetric), its braiding (symmetry) lifts to
$\Comon_\star(\ca{V})$.

\subsection{Measuring morphisms in duoidal categories}
\label{sec:meas-morph-duoid}
In this section $(\ca{V},\diamond,i,\star,j)$ wil be a duoidal category.
\begin{df}
  \label{df:2}
  Let $(a,\eta^a,\mu^a)$ and $(b,\eta^b,\mu^b)$ be $\star$-monoids and
  $(c,\varepsilon,\delta)$ be a $\star$-comonoid in $\mathcal{V}$. A morphism
  $\varphi{}\colon c\diamond a\to b$ \emph{measures} (on the right), or is a (right)
  \emph{measuring morphism} from $a$ to $b$, if the following diagrams commute.
  \begin{equation}
    \label{eq:8}
    \begin{tikzcd}[column sep=small]
      c\diamond (a\star a)
      \ar[r,"\delta\diamond 1"]
      \ar[d,"1\diamond \mu^a"']
      &
      (c\star c)\diamond (a\star a)
      \ar[r,"\zeta"]
      &
      (c\diamond a )\star(c\diamond a)
      \ar[d,"\varphi{}\star \varphi{}"]
      \\
      c\diamond a
      \ar[r,"\varphi{}"]
      &
      b
      &
      b\star b
      \ar[l,"\mu^b"']
    \end{tikzcd}
    \qquad
    \begin{tikzcd}[column sep=small]
      c\diamond j
      \ar[r,"\varepsilon\diamond 1"]
      \ar[d,"1\diamond \eta^a"']&
      j\diamond j\ar[r,"\mu_j"]
      &
      j
      \ar[d,"\eta^b"]
      \\
      c\diamond a\ar[rr,"\varphi{}"]
      &&
      b
    \end{tikzcd}
  \end{equation}
\end{df}
When the duoidal category is induced by a braided monoidal category, so
$\diamond=\star$ and $i=j$, then a measuring morphism is the obvious
generalisation of the measurings introduced in~\cite[Ch.~VII]{Sweedler}.

Given two $\star$-monoids $a$ and $b$ as in Definition~\ref{df:2}, there is a
functor $\mathrm{Meas}(\mi;a,b)\colon \Comon_\star(\ca{V})^{\mathrm{op}}\to\Set$ 
that sends a $\star$-comonoid $c$ to the set $\mathrm{Meas}(c;a,b)$ of measuring morphisms
$\varphi{}\colon c\diamond a\to b$, and a comonoid
morphism $g$ to the function $\varphi\mapsto\varphi{}\cdot(g\diamond 1_a)$.
\begin{df}
  \label{df:3}
  A \emph{universal measuring morphism} for a pair of $\star$-monoids $a$,
  $b$, is a $\star$-comonoid $P(a,b)$ together with a measuring morphism
  $\phi(a,b)\colon P(a,b)\diamond a \to b$ that is a representation of the
  functor $\mathrm{Meas}(\mi;a,b)$ of the previous paragraph. The comonoid $P(a,b)$ is
  called the \emph{universal measuring comonoid} of the pair $a$, $b$.
\end{df}
Another way of expressing this definition is that each measuring
morphism $\varphi\colon c\diamond a\to b$ factors as
$\phi(a,b)\cdot (g\diamond 1_{a})$ for a unique morphism of $\star$-comonoids
$g\colon c\to P(a,b)$; namely $P(a,b)$ is terminal in an appropriate category of measuring comonoids.

Subject to the existence of a $\diamond$-monoidal closed structure, we can now mimick the arguments that lead to the previous \cref{eq:actioncomon} 
and establish the following.

\begin{prop}\label{prop:duoidalaction}
Suppose the duoidal category $(\ca{V},\diamond,I,\star,J)$ is right
$\diamond$-closed. Then, the right internal hom $[\mi,\mi]^\diamond$ lifts to a
functor
\begin{equation}\label{eq:diamondhom}
 [\mi,\mi]^\diamond\colon\Comon_\star(\ca{V})^\op\times\Mon_\star(\ca{V})\to\Mon_\star(\ca{V})
\end{equation}
which carries the structure of an action of the
$\diamond$-monoidal category $\Comon_\star(\ca{V})^{\mathrm{op,rev}}$ on the category $\Mon_\star(\ca{V})$.
\end{prop}

\begin{proof}
  We shall write $[a,b]$ for $[a,b]^\diamond$ in this proof.
 Since $\diamond\colon(\ca{V},\star,J)\times(\ca{V},\star,J)\to(\ca{V},\star,J)$ is a oplax monoidal functor by \cref{def:duoidal}, its right 
parametrised adjoint $[\mi,\mi]\colon(\ca{V},\star,J)^\op\times(\ca{V},\star,J)\to(\ca{V},\star,J)$ is lax monoidal by
\cref{prop:doctrinalparameter}. Therefore it lifts to a functor between the respective categories of (co)monoids as in \cref{eq:diamondhom},
bearing in mind that $\Mon_\star(\ca{V}^\op)\cong\Comon_\star(\ca{V})^\op$ as is
the case for any monoidal category.

Furthermore, the canonical natural ismorphism $[b,[a,c]]\cong [a\diamond b,c]$ is
the mate of the associativity natural isomorphism $(a\diamond b)\diamond c\cong
a\diamond(b\diamond c)$, which is a $\star$-monoidal natural transformation by
the definition of duoidal category. Then, $[b,[a,c]]\cong [a\diamond b,c]$ is also a
monoidal transformation betwen lax $\star$-monoidal functors by \cref{prop:doctrinalparameter}, and, therefore,
lifts to the respective categories of monoids, providing an associativity
constraint for the action of of the statement.
A similar argument shows that the natural isomorphism $[i,b]\cong b$ is
monoidal, providing a unit constraint for the action.
\end{proof}

The above functor generalizes the `convolution' product of maps between a coalgebra and an algebra in the duoidal setting:
the multiplication $[c,b]^\diamond\star[c,b]^\diamond\to[c,b]^\diamond$ and unit $j\to[c,b]^\diamond$ on some $[c,b]^\diamond$ for $c$ a $\star$-comonoid and $b$ a $\star$-monoid are formed as the $\diamond$-adjuncts of
\begin{equation}
  \label{eq:duoidalconvolution}
    c\diamond
    ([c,b]^\diamond\star[c,b]^\diamond)
    \xrightarrow{\delta\diamond 1}
    (c\star c)\diamond([c,b]^\diamond\star[c,b]^\diamond)
    \xrightarrow{\zeta}
    (c\diamond [c,b]^\diamond)\star(c\diamond [c,b]^\diamond)
    \xrightarrow{\mathrm{ev}\star\mathrm{ev}}
    b\star b\xrightarrow{\mu^b}
    b
\end{equation}
 and $j\diamond c\cong c\xrightarrow{\epsilon}j\xrightarrow{\eta}b$.

The proof of the following lemma, which gives the intution behind the measuring morphism \cref{df:2}, is a straightforward application of the
tensor-hom adjunction and left to the reader.
\begin{lemma}
  \label{l:3}
  Suppose that the duoidal category $\mathcal{V}$ is right $\diamond$-closed.
  Given $\star$-monoids $a$ and $b$, and a $\diamond$-comonoid $c$, consider a
  morphism $f\colon c\diamond a\to b$ and its transpose $\hat f\colon a\to
  [c,b]^{\diamond}$. Then, $f$ is a measuring morphism if and only if $\hat f$
  is a morphism of monoids.
\end{lemma}
If the universal measuring comonoid $P(a,b)$ exists, then there is a natural
isomoprhism
\begin{equation}
  \label{eq:10}
  \Comon_\star(c,P(a,b))\cong\mathrm{Meas}(c;a,b)\cong
  \Mon_\star(a,[c,b]^{\diamond}).
\end{equation}
Therefore, the assignment $(a,b)\mapsto P(a,b)$, if it exists, will provide a
parametrised adjoint to $[\mi,\mi]^\diamond$. This is next section's subject
of study.

\subsection{Enrichment of monoids in comonoids}
\label{sec:enrichm-mono-comon}

We now wish to extend \cref{thm:measucoalg,thm:STmoncats} in the context of duoidal categories. Recall that a functor is called accessible when it preserves filtered colimits.

\begin{thm}\label{thm:measuduoidal}
Suppose $(\ca{V},\diamond,i,\star,j)$ is a duoidal category whose monoidal
structure $(\ca{V},\diamond,i)$ is biclosed, with right internal hom
$[\mi,\mi]^\diamond$. Suppose further that $\mathcal{V}$ is a locally
presentable category and $\star$ is an accessible functor. Then, each functor $$([\mi,b]^\diamond)^\op\colon\Comon_\star(\ca{V})\to\Mon_\star(\ca{V})^\op$$ has a 
right adjoint $P(\mi,b)$ and each functor $$[c,\mi]^\diamond\colon\Mon_\star(\ca{V})\to\Mon_\star(\ca{V})$$ has a left adjoint 
$c\triangleright^\diamond\mi$, via natural bijections
\begin{equation}\label{eq:duoidalhomPtriangle}
  \Comon_\star(\ca{V})(c,P(a,b))\cong\Mon_\star(\ca{V})(a,[c,b]^\diamond)
  \cong \Mon_\star(\ca{V})(c\triangleright^\diamond a,b)
\end{equation}
\end{thm}

\begin{proof}
By \cref{prop:MonComonproperties}, since the monoidal product $\mi\star\mi$ preserves filtered colimits, both categories $\Mon_\star(\ca{V})$ and 
$\Comon_\star(\ca{V})$ are themselves locally presentable therefore complete, cocomplete and with a small dense subcategory. As a result, sheer 
cocontinuity of $([\mi,b]^\diamond)^\op$ will be enough to establish the first adjunction, whereas for the second one a more classical approach will 
suffice.

In the commutative diagram displayed on the left below, both vertical functors
are comonadic and the bottom functor is cocontinuous, being a left adjoint
$[\mi,b]^\op\dashv[\mi,b]$ due to biclosedness. 
As a result, the top one is also cocontinuous, thus
has a right adjoint $P(\mi,b)$. Here we have used
Proposition~\ref{prop:MonComonproperties}.
\begin{displaymath}
 \begin{tikzcd}[column sep=.7in]
\Comon_\star(\ca{V})\ar[r,"{([\mi,b]^\diamond)^\op}"]\ar[d] & \Mon_\star(\ca{V})^\op\ar[d] \\
\ca{V}\ar[r,"{([\mi,b]^\diamond})^\op"] & \ca{V}^\op
\end{tikzcd}
\qquad
 \begin{tikzcd}[column sep=.7in]
\Mon_\star(\ca{V})\ar[r,"{[c,\mi]^\diamond}"]\ar[d] & \Mon_\star(\ca{V})\ar[d] \\
\ca{V}\ar[r,"{[c,\mi]^\diamond}"] & \ca{V}
 \end{tikzcd}
\end{displaymath}
Similarly, both vertical functors in the commutative diagram on the right are
monadic, while the the bottom functor has a left adjoint
$c\diamond \mi\dashv[c,\mi]^\diamond$. Dubuc's Adjoint Triangle Theorem applies
to give also a top left adjoint $c\triangleright^\diamond\mi$, since
$\Mon_\star(\ca{V})$ has coequalizers by
Proposition~\ref{prop:MonComonproperties}.
\end{proof}

Combining the above \cref{thm:measuduoidal} with the fact that
$[\mi,\mi]^\diamond$ is an action by \cref{prop:duoidalaction}, we are in the
position of applying \cref{prop:4} to provide an enrichment of $\star$-monoids in $\star$-comonoids.

\begin{thm}\label{thm:sweedlerduoidal}
Suppose $(\ca{V},\diamond,i,\star,j)$ is duoidal category
such that  $\mathcal{V}$ a locally presentable category, $\star$ is an
accessible functor and the monoidal structure $(\diamond,i)$ is biclosed. Then, there exists a tensored and cotensored
$\Comon_\star(\mathcal{V})$-category whose underlying category is (isomorphic
to) $\Mon_\star(\ca{V})$.
\end{thm}
\begin{proof}
In order to speak of tensors and cotensors, this lifted $\diamond$-monoidal structure on
$\Comon_\star(\mathcal{V})$ should be biclosed. This is essentially showed
in~\cite[\S 3.2]{MonComonBimon} using the locally presentability and
accessibility hypotheses (although the reference uses finitary functors);
see also the comments after~\cite[Prop.~2.9]{Measuringcomonoid}.
We can now apply Corollary~\ref{cor:5} to the action \cref{eq:diamondhom} of
$\Comon_\star(\mathcal{V})^{\mathrm{op,rev}}$
with parametrised adjoints $P$ and
$\triangleright^\diamond$ as in \cref{eq:duoidalhomPtriangle}.
\end{proof}
The enriched category of the theorem has enriched homs $P(a,b)$, for
$\star$-monoids $a$, $b$. Given a $\star$-comonoid $c$ and a $\star$-monoid $a$,
their tensor product is $c\triangleright^\diamond a$, while their cotensor
product is $[c,a]^\diamond$. The functors $P$ and $\triangleright^\diamond$ which generalize \cref{eq:monoidalsweedler}
in the duoidal setting, can also be called Sweedler hom and Sweedler product.
 
\begin{rmk}
  \label{rmk:1}
  The functor $\triangleright^\diamond$ has the structure of a closed action of
  $\Comon_\star(\mathcal{V})$ on $\Mon_\star(\mathcal{V})$, induced from the
  action structure of $[\mi,\mi]^\diamond$. As noted in Corollary~\ref{cor:5},
  the enriched category associated to $\triangleright^\diamond$ is equal to the
  enriched associated to $[\mi,\mi]^\diamond$. This means that there really is
  only one possible choice of $\Comon_\star(\mathcal{V})$-category in
  Theorem~\ref{thm:sweedlerduoidal}.
\end{rmk}

\section{Graded objects and species}
\label{sec:grad-objects-spec}

In this section we brifely review the most common monoidal structures on the
categories of $\mathbb{N}$-graded objects and species. We shall denote by
$\mathbb{N}$ and $\mathbb{P}$ the subcategories of the category of sets whose
objects are the finite cardinals and with the following morphisms: only identity
functions in the case of $\mathbb{N}$, and permutations in the case of
$\mathbb{P}$. Thus, $\mathbb{N}$ is discrete category of natural numbers and
$\mathbb{P}$ can be regarded as the disjoint union of the permutation groups
$P_n$ (if one sets $P_0=1$).

The categories $\Grd(\mathcal{V})=[\mathbb{N},\mathcal{V}]$ is usually called
the category of \emph{$\mathbb{N}$-graded objects} in $\mathcal{V}$, while
$\Sp(\mathcal{V})=[\mathbb{P},\mathcal{V}]$ is called the category of
\emph{species} in $\mathcal{V}$, or the category of \emph{symmetric collections}
in $\mathcal{V}$ as in \cite{KellyonMay}.
There is another presentation of $\Sp(\mathcal{V})$ that replaces $\mathbb{P}$
by the equivalent category of finite sets. This is the preferred approach in
\cite{FoncteursAnalytiques} and \cite{Species}.

The interest on $\Grd(\mathcal{V})$ and $\Sp(\mathcal{V})$ stems from their
power to express
combinatorial structures, thanks to the existence of a \emph{substitution}
monoidal structure on each one of them (if $\mathcal{V}$ has the sufficient
extra structure). There are several ways of describing these structures, ranging
from the very explicit, low technology, definitions to the very sophisticated
involving 2-monads and Kleisli bicategories. We will have more to say about
substitution in Section~\ref{sec:subst-mono-prod}.

In the section that follows, we explore four different monoidal structures on the categories of graded objects and species in $\ca{V}$: the pointwise or Hadamard $\ot$, the Cauchy $\bullet$, the $*$ and the substitution $\circ$. The former three are all examples of Day convolution \cref{eq:dayconv} in the specific context, see also \cite[\S~2.3--5]{BelgianPaper}.

In some combinations, certain products in fact form duoidal structures, some of them resulting from \cref{prop:duoidalpresheaves}. Therefore, the theory described in \cref{sec:SweedlerTheoryduoid} applies to provide a number of enrichments of monoids in comonoids with respect to different monoidal products. Below, we provide a summary of the main results of this chapter.

\begin{thm}\label{thm:allduoidalstructures}
 Suppose $\ca{V}$ is a symmetric monoidal category with coproducts which are preserved by $\ot$ in both variables.
 \begin{enumerate}
  \item $(\VGrd,\bullet,\ot)$ and $(\VSp,\bullet,\ot)$ are duoidal categories (\cref{prop:duoidVGrd1,prop:duoidVSym1});
  \item $(\VGrd,\ot,\bullet)$ and $(\VSp,\ot,\bullet)$ are duoidal categories, if $\ca{V}$ has moreover finite biproducts  (\cref{prop:duoidVGrd1',prop:duoidVSym1});
  \item $(\VGrd,\circ,\ot)$ is a duoidal category (\cref{prop:duoidVGrd2}), and so is $(\VSp,\circ,\ot)$ if $\ca{V}$ has moreover all colimits and are preserved by $\ot$ (\cref{prop:duoidVSym2});
  \item $(\VGrd_+,\ot,\circ)$ and $(\VSp_+,\ot,\circ)$ are duoidal categories, if $\ca{V}$ has moreover finite biproducts (\cref{prop:duoidVGrd+});
 \end{enumerate}
\end{thm}

\begin{thm}\label{thm:allenrichments}
Suppose $\ca{V}$ is a symmetric monoidal closed and locally presentable category.
\begin{enumerate}
  \item $\Grd(\Mon(\ca{V}))$ is tensored and cotensored enriched in $(\Grd(\Comon(\ca{V})),\ot,\I)$, as well as $\Sp(\Mon(\ca{V}))$ in $(\Sp(\Comon(\ca{V})),\ot,\I)$ (\cref{thm:HadGrdEnr});
  \item $\nc{gMon}(\ca{V})$ is tensored and cotensored enriched in $(\nc{gComon}(\ca{V}),\bullet,\1)$ (\cref{thm:CauchyGrdEnr});
  \item $\nc{TwMon}(\ca{V})$ is tensored and cotensored enriched in $(\nc{TwComon}(\ca{V}),\bullet,\1)$ (\cref{thm:CauchySymEnr});
  \vspace{.1in}\hrule\vspace{.1in}
  \item $\Grd(\Mon(\ca{V}))$ is tensored and cotensored enriched in $(\Grd(\Comon(\ca{V})),\bullet,\1)$ (\cref{thm:CauchyHadGrdEnr});
  \item $\nc{gMon}(\ca{V})$ is tensored and cotensored enriched in $(\nc{gComon}(\ca{V}),\ot,\I)$, if $\ca{V}$ has moreover finite biproducts (\cref{thm:HadCauchyGrdEnr});
  \item $\Sp(\Mon(\ca{V}))$ is tensored and cotensored enriched in $(\Sp(\Comon(\ca{V})),\bullet,\1)$ (\cref{thm:CauchyHadSymEnr});
   \item $\nc{TwMon}(\ca{V})$ is tensored and cotensored enriched in $(\nc{TwComon}(\ca{V}),\ot,\I)$, if $\ca{V}$ has moreover finite biproducts (\cref{thm:CauchyHadSymEnr});
  \item $\nc{(s)Opd}_+(\ca{V})$ is tensored and cotensored enriched in $(\nc{(s)Coopd}_+(\ca{V}),\ot,\I)$, if $\ca{V}$ has moreover finite biproducts (\cref{thm:HadSubGrdEnr});
  \item $\nc{sOpd}(\ca{V})$ is tensored and cotensored enriched in $(\nc{sCoopd}(\ca{V}),*,\J)$, if $\ca{V}$ is cartesian monoidal (\cref{thm:starSubSymEnr}).
 \end{enumerate}
\end{thm}

\subsection{Pointwise or Hadamard tensor product}
\label{sec:pointw-or-hadam}

Before discussing the convolution tensor produts on the categories of graded
objects and species, we recall that these categories, as any functor category, have a monoidal structure defined pointwise from the
structure of $\mathcal{V}$: for two species or graded objects,
\begin{equation}\label{eq:Had}
 (V\otimes W)_n:= V_n\otimes W_n
\end{equation}
The unit the constant collection on $I$, $\B{I}=\{I\}_{n\in\mathbb{N}}$, and $(\VGrd,\otimes,\I)$ and $(\VSym,\otimes,\B{I})$ are braided when $\ca{V}$ is, and monoidal closed when $\ca{V}$ is via $[V,W]^\ot(n)=[V_n,W_n]$. 

For any category $\ca{J}$, there are isomorphisms between
$\Mon[\mathcal{J},\mathcal{V}]$ and $[\mathcal{J},\Mon(\mathcal{V})]$, and between
$\Comon[\mathcal{J},\mathcal{V}]$ and $[\mathcal{J},\Comon(\mathcal{V})]$. These
express the fact that a (co)monoid in $[\mathcal{J},\mathcal{V}]$ is a functor
$J\colon\mathcal{J}\to \mathcal{V}$ with (co)monoid structures on each object
$J(j)\in\mathcal{V}$, and such that each morphism $J(f)$ is a morphism of
(co)monoids.
For each object $j\in \mathcal{J}$ the ``evaluation at $j$'' functor $E_j\colon
[\mathcal{J},\mathcal{V}]\to\mathcal{V}$ is strict monoidal

The enrichment of $\Grd(\Mon(\mathcal{V}))$ and $\Sp(\Mon(\mathcal{V}))$ in
$\Grd(\Comon(\mathcal{V}))$ and $\Sp(\Comon(\mathcal{V}))$, repespectively,
provided by Theorem~\ref{thm:STmoncats} are related to the enrichments of
$\Mon(\mathcal{V})$ in $\Comon(\mathcal{V})$, via the evaluation functors
\begin{equation}
  \label{eq:17}
  E_n\colon\Grd(\mathcal{V})\to\mathcal{V}\leftarrow\Sp(\mathcal{V})\colon E_n
\end{equation}
as given by the following result.

\begin{cor}\label{thm:HadGrdEnr}\label{thm:HadSymEnr}
  Suppose $\ca{V}$ is a locally presentable, braided monoidal closed
  category. The category $\Grd(\Mon(\ca{V}))$ of graded objects in
  $\Mon(\ca{V})$ is tensored and cotensored enriched in the monoidal
  biclosed category $\Grd(\Comon(\ca{V}))$ of graded objects in $\Comon(\ca{V})$
  with the pointwise tensor product, and so is the category $\Sp(\Mon(\ca{V}))$
  of species in $\Mon(\ca{V})$ in $\Sp(\Comon(\ca{V}))$ of species in
  $\Comon(\ca{V})$.  Furthermore, for each finite cardinal $n$, the evaluation
  functors~\eqref{eq:17} induce full and faithful $\Comon(\mathcal{V})$-functors
  \begin{equation}
    \label{eq:18}
    \Grd(\Mon(\mathcal{V}))\to\Mon(\mathcal{V})\leftarrow \Sp(\Mon(\mathcal{V}))
  \end{equation}
\end{cor}
Even though the proof of this corollary will be given in greater generality in
Appendix~\ref{sec:pointw-mono-struct}, we shall say a few explanatory words about
the statement. For simplicity, we shall speak of the case of species, as the
case of graded objects is identical. Let us denote by $\mathsf{M}$ the
$\Sp(\Comon(\mathcal{V}))$-category with underlying category
$\Sp(\Mon(\mathcal{V}))$ given by Theorem~\ref{thm:STmoncats}. Similarly, let
$\mathsf{A}$ be the $\Comon(\mathcal{V})$-category whose underlying category is
$\Mon(\mathcal{V})$. The strict monoidal evaluation functor
$E_n\colon\Sp(\Comon(\mathcal{V}))\to\Comon(B)$ sends a species of comonoids
$\{c_i\mid i\in \mathbb{N}\}$ to the comonoid $c_n$. The corollary says that the
$n$th component of the species of comonoids $\mathsf{M}(A,B)$ is just
$\mathsf{A}(A_n,B_n)$, for a pair of species of monoids $A$, $B$.

\subsection{Convolution structures}
\label{sec:conv-struct}

Convolution monoidal structures were briefly described below
Proposition~\ref{prop:duoidalpresheaves} as left Kan extensions. Let's assume
that $\mathcal{V}$ is a monoidal category. For any monoidal structure
on $\mathbb{N}$, the left Kan extensions that define the induced convolution
tensor product and unit object in $\Grd(\mathcal{V})$ exist when
$\mathcal{V}$ has coproducts, since $\mathbb{N}$ is discrete. The associativity
and unit constraints will exist if the tensor product in $\mathcal{V}$ preserves
coproducts.

The finite cardinal addition and product correspond to finite coproducts and
finite cartesian products in the category of sets. This gives rise to a pair of
monoidal structures on each of $\mathbb{N}$ and $\mathbb{P}$, addition and
product.

The convolution monoidal structure $\Grd(\mathcal{V})$ and $\Sp(\mathcal{V})$
induced by addition, known as the \emph{Cauchy} tensor product, is the usual
tensor product of graded objects. We will denote it by the symbol $\bullet$ and
has formulas, for graded objects on the left and species on the right,
\begin{equation}
  \label{eq:CauchySym}
  (V\bullet W)_n\cong
  \sum_{k+m=n}V_k\otimes W_{m}
  \quad
  (A\bullet B)_n
  \cong
  \int^{k,m} \mathbb{P}(k+m,n)\cdot A_k\otimes B_m\cong \sum\limits_{\substack{k+m=n \\ \textrm{Sh}(k,m)}}A_k\otimes B_m
\end{equation}
where the set $\textrm{Sh}(k,m)$ is of all $(k,m)$-shuffles, i.e. ways of forming a totally ordered set with size $m+k$ out of two totally ordered sets of sizes $m$ and $k$.
In both cases the unit object for the Cauchy tensor product is the object
$\mathbf{1}$ with component all components equal to the initial object $0$,
except $\mathsf{1}_0=I$, the unit object of $\mathcal{V}$. For species, see \cite[\nopp 2.2]{KellyonMay} or \cite[\S~ 5.1.4]{AlgebraicOperads} for vector spaces. 

\begin{rmk}
  \label{rmk:4}
  \begin{enumerate}[wide]
  \item The existence of the Cauchy monoidal structure on $\Grd(\mathcal{V})$
    depends only of the existence of countable coproducts in $\mathcal{V}$ and
    the preservation of these by the tensor product in each variable.
  \item The existence of the Cauchy monoidal structure on $\Sp(\mathcal{V})$ is
    guaranteed by the existence of countable coproducts and quotients by group
    actions in $\mathcal{V}$ and their preservation by the tensor product in
    each variable. This strong requirement is usually replaced by the simpler
    hypothesis that $\mathcal{V}$ should be cocomplete and the tensor product
    should be cocontinuous in each variable.
  \end{enumerate}
\end{rmk}

The Cauchy monoidal structure is left (right) closed if $\mathcal{V}$ is
complete and left (right) closed, with internal homs given by particular
instances of the formulas~\eqref{eq:dayhom}. In the case of $\Grd(\mathcal{V})$,
only products are needed and the left internal hom has an expression
$[V,W]^\bullet_{\ell,n}\cong\prod_{i\geq 0}[V_i,W_{i+n}]_\ell$; similarly for
the right internal hom. In the case of $\VSp$, ends are needed and $[A,B]^\bullet_{\ell,n}=\int_{b}[A_b,B_{b+n}]_\ell$.

Braidings for the monoidal category $\mathcal{V}$ induce braidings for
$\Grd({\mathcal{V}})$ and $\Sp(\mathcal{V})$. Indeed, we have already described
how to induce a braiding on a convolution product on $[\mathcal{A},\mathcal{V}]$
from a braiding on $\mathcal{A}$ and another on $\mathcal{V}$.
In the present situation, $\mathcal{A}$ is either $\mathbb{N}$ or $\mathbb{P}$
equipped with the canonical symmetry of the coproduct monoidal structure. In
graded objects, the
resulting brading has components $(V\bullet W)_n\cong (W\bullet V)_n$ that are
just braiding $V_k\bullet W_m\cong W_m\bullet V_k$ in $\mathcal{V}$.  These are
not the only possible bradings. In fact, the most commonly used bradings in the
category of graded vector spaces has components $(v\otimes w)\mapsto q^{km}
w\otimes v$ if $v\in V_k$ and $w\in W_m$, for $q\neq 0$. Most often than not
$q=-1$. See for example~\cite[\S 2.3.1]{Species}.

Monoids and comonoids for the Cauchy tensor product in $\Grd(\mathcal{V})$ are
called \emph{graded monoids} and \emph{graded comonoids}. These are the direct
generalisations to monoidal categories of the graded algebras and graded
coalgebras so common in algebra. A graded monoid, then, consists of
a graded object $V$ with multiplication and unit components
\begin{equation}
  \label{eq:23}
  \mu_{k,m}\colon V_k\otimes V_m\to V_{k+m}\qquad I\to V_0
\end{equation}
satisfying associativity and unit axioms. A comonoid is a graded object $Z$ with a
comultiplication and counit components
\begin{equation}
  \label{eq:gradedcomonoid}
  \delta_{n}\colon Z\to\sum_{k+m=n}Z_k\otimes Z_m\qquad
  \varepsilon_n\colon Z_n\to \mathbf{1}_n=
  \begin{cases}
    I&n=0\\0&n\neq 0
  \end{cases}
\end{equation}
and satisfying coassociativity and counit axioms.
We denote these categories by $\mathsf{gMon}(\ca{V})=\Mon_\bullet(\VGrd)$ and $\mathsf{gComon}(\ca{V})=\Comon_\bullet(\VGrd)$.

\begin{rmk}
  \label{rmk:7}
  In the classical setting, the category $\ca{V}$ has
  finite biproducts, for example vector spaces. Then, the components of the
  comultiplication land on a (finite) product, therefore can be expressed as
  $\delta_{m,k}\colon V_{m+k}\to V_m\ot V_k$ for all $m,k$, and the counit is
  only $V_0\to I$ (since the rest are zero maps).
\end{rmk}

\cref{thm:STmoncats} yields the following enrichment result.
\begin{cor}\label{thm:CauchyGrdEnr}
  Suppose $\ca{V}$ is braided monoidal closed and locally presentable. Then
  the category of graded monoids $\mathsf{gMon}(\ca{V})$ is tensored and
  cotensored enriched in the monoidal biclosed category of graded
  comonoids $(\mathsf{gComon}(\ca{V}),\bullet,\1)$.
\end{cor}

If $Z$ is a graded comonoid and $V$ is a graded monoid in $\ca{V}$, then the
multiplication of the graded monoid $[Z,V]^\bullet$ corresponds to the graded
morphism $Z\bullet [Z,V]^\bullet \bullet[Z,V]^\bullet\to V$ with components
\begin{equation}\label{eq:convCauchy}
 \begin{tikzcd}[row sep=.1in,column sep=.2in]
   \left(Z\bullet [Z,V]^\bullet\bullet[Z,V]^\bullet\right)_n
   \ar[d,equal]\ar[dddrr,dashed,bend left] &&
   \\
   \displaystyle\sum_{k+m+\ell=n}
   Z_\ell\otimes \prod_{i}\otimes
   [Z_i,V_{i+k}]\ot
   \prod_{j}[Z_j,V_{j+m}]
   \ar[d,"\delta\otimes 1\ot1\delta"'] &&
   \\
   \displaystyle
   \sum_{p+q=\ell} Z_p\ot Z_q
   \otimes
   \sum_{k+m+\ell=n}
   \prod_{i}
   [Z_i,V_{i+k}]\ot
   \prod_{j}[Z_j,V_{j+m}]\ot
   \ar[d,"1\ot\beta\ot1"' ] &&
   \\
   \displaystyle\sum_{k+m+p+q=n}
   Z_p\otimes
   \prod_{i}[Z_i,V_{i+k}]\ot
   Z_q\otimes \prod_{j}[Z_j,V_{j+m}]
   \ar[r,"\mathrm{ev}\ot\mathrm{ev}"']
   &
   \displaystyle\sum_{k+m+p+q=n}V_{p+k}\ot V_{q+m}\ar[r,"\mu"'] & V_n
\end{tikzcd}
\end{equation}
\begin{rmk}
  The enrichment of Corollary~\ref{thm:CauchyGrdEnr} is closely related to the
  enrichment of differential graded
  algebras in differential graded coalgebras of \cite[Thm.~0.0.1]{AnelJoyal}.
  In this setting, $\mathcal{V}=\Vect$ and for $f$ a $k$-degree graded map,
  $g$ an $m$-degree graded map, $c$ a $(p+q)$-homogeneous element,
  \begin{displaymath}
    (f\cdot g)(c):=f(c^{(1)})g(c^{(2)})\in V_{k+m+p+q}
  \end{displaymath}
  depending on the choice of symmetry; see \cite[\S~3.2]{AnelJoyal}.
\end{rmk}

Monoids and comonoids in $(\VSp,\bullet,\1)$ are of central importance for the
work of \cite{Species}, and are sometimes called `twisted (co)algebras' for
$\ca{V}=\Vect$ introduced in \cite{TwistedAlgs} and appear in various works like
\cite{FoncteursAnalytiques,Twisteddescentalg}.
a \emph{twisted monoid} is a graded monoid \cref{eq:23} with symmetric actions on components, where 
the multiplication maps $\mu_{m,k}$ are $P_m\times P_k$ equivariant.
We denote by $\nc{TwMon}(\ca{V})=\Mon_\bullet(\VSp)$ and
$\nc{TwComon}(\ca{V})=\Comon_\bullet(\VSp)$.

Under the usual assumptions, \cref{thm:STmoncats} once again induces an
enrichment of monoids in comonoids in $\Sp(\ca{V})$ with the Cauchy product.

\begin{cor}\label{thm:CauchySymEnr}
  In a locally presentable, braided monoidal closed category $\ca{V}$,
  $\nc{TwMon}(\ca{V})$ is tensored and cotensored enriched in the
  monoidal biclosed category $(\nc{TwComon}(\ca{V}),\bullet,\1)$.
\end{cor}

If $C$ is a twisted comonoid and $A$ is a twisted monoid, $[C,A]^\bullet$
becomes a twisted monoid with multiplication formed very much as
in~\cref{eq:convCauchy}.

Finally, the monoidal structure on $\mathbb{N}$ and $\mathbb{P}$ given by the cartesian
product gives rise to convolution monoidal structures on $\Grd(\mathcal{V})$ and
$\Sp(\mathcal{V})$. All the considerations we have made for the Cauchy tensor
product hold for this new convolution structure, replacing addition by product
when necessary. 

\subsection{Substitution monoidal product}
\label{sec:subst-mono-prod}

Suppose that $(\ca{V},\ot,I)$ is a monoidal category with coproducts preserved by $\otimes$ in both variables.
The $m$-fold Cauchy tensor product is a graded object
\begin{equation}\label{eq:mCauchyGrd}
 V^{\bullet m}_n=\overbrace{(V\bullet V\bullet\ldots\bullet V)_n}^{m\textrm{-times}}=
 \sum_{n_1+\ldots+n_m=n}V_{n_1}\ot\ldots\ot V_{n_m}
\end{equation}
where the sum is over all possible combinations of $m$ natural numbers (possibly zero) whose addition is equal to $n$. The substitution monoidal product for two $\ca{V}$-graded objects is then defined by
\begin{equation}\label{eq:subVGrd}
 (V\circ W)_n:=\sum_{m\geq0}V_m\ot W^{\bullet m}(n)\cong\sum_{\substack{m\geq0 \\ n_1+\ldots+n_m=n}}V_m\ot W_{n_1}\ot\ldots\ot W_{n_m}.
\end{equation}
Notice that even though for each $m$ the second sum is finite, the first sum is over all natural numbers. The unit for this tensor is the graded objet $\X$ with
\begin{equation}\label{eq:unitsub}
\X_n=\begin{cases}
    I,& n=1\\
    0,& n\neq1
   \end{cases}
\end{equation}
which is similar, but different, to the Cauchy monoidal unit $\1$. If $\ca{V}$ is braided, this tensor product on $\VGrd$ does not inherit the braiding: it is clear that the roles of $V$ and $W$ are not interchangeable. Moreover, when $\ca{V}$ is left closed and products exist, $(\VGrd,\circ,\X)$ is left closed via $\mi\circ V\dashv[V,\mi]_\ell^\circ$ where
\begin{equation}\label{eq:subclosedVGrd}
 [V,W]^\circ_{\ell,n}:=\prod_{m\geq0}[V^{\bullet n}_m,W_m]_\ell
\end{equation}
but even if $\ca{V}$ is right closed, $\VGrd$ is not.

\begin{rmk}
The substitution product can be viewed as a composite of the pointwise and Cauchy product as follows. First of all, the mapping $(W,m)\mapsto W^{\bullet m}$ corresponds to a functor $\VN\times\ca{N}\to\VN$, so denote its transpose functor $(\mi)^\bullet\colon\VN\to[\ca{N},\VN].$ 
Moreover, the Hadamard functor is in fact an instance of a class of functors induced by the underlying tensor product in $\ca{V}$, expressed as 
\begin{equation}\label{eq:canonical}
\ot\colon\VN\times\VN\xrightarrow{\textrm{cano}}(\ca{V}\times\ca{V})^{\ca{N}\times\ca{N}}\xrightarrow{\ot_\ca{V}}\ca{V}^{\ca{N}\times\ca{N}}\xrightarrow{1^\Delta}\VN
\end{equation}
by $(\{V_n\},\{W_{n'}\})\mapsto\{(V_n,W_{n'})\}\mapsto\{V_n\ot W_{n'}\}\mapsto\{V_n\ot W_n\}$. Similarly, we can build
$$\wt{\ot}\colon\VN\times[\ca{N},\VN]=(\VN)^1\times(\VN)^\ca{N}\xrightarrow{\textrm{cano}}(\VN\times\VN)^\ca{N}\xrightarrow{\ot}(\VN)^\ca{N}=[\ca{N},\VN]
$$
which performs $(\{V_n\},\{F_{n'}(\mi)\})\mapsto\{V_n\ot F_{n}(\mi)\}$ producing a diagram of shape $\ca{N}$ in $\VN$. Finally, the substitution in $\VGrd$ can be expressed as the composite
\begin{equation}\label{eq:subwithHadCau}
 \begin{tikzcd}[row sep=.3in]
\VN\times\VN\ar[r,dashed,"\circ"]\ar[d,"1\times(\mi)^\bullet"'] & \VN \\
\VN\times[\ca{N},\VN]\ar[r,"{\wt{\ot}}"'] & {}[\ca{N},\VN]\ar[u,"\sum_{m\geq0}"']
 \end{tikzcd} 
\end{equation}
where the last functor essentially takes some $G\colon\ca{N}\to\VN$ to its colimit (since $\ca{N}$ is discrete).
\end{rmk}

Since $(\VGrd,\circ,\X)$ is not braided, it does not satisfy the assuptions of \cref{thm:STmoncats} like the previous monoidal structures. Therefore there is no induced enrichment of its category of monoids in the category of comonoids in that way; notice that the latter is not even a $\circ$-monoidal category. However, in the next section we will establish a different enrichment between these structures, which are of central importance in numerous areas of mathematics and we next discuss.

A monoid in $(\VGrd,\circ,\X)$ is a \emph{non-symmetric}, or \emph{non-}$\Sigma$, or \emph{planar operad}. Explicitly, it is a graded object $\{V_n\}_{n\geq0}$ which comes equipped with multiplication $\mu\colon V\circ V\to V$ and unit $\eta\colon\X\to V$ whose components at each $n\in\ca{N}$ are, for all $m,n_i\geq0$,
\begin{equation}\label{eq:nonsymoperad}
\mu_{m,n_i}\colon V_m\ot V_{n_1}\ot\ldots\ot V_{n_m}\to V_{n_1+\ldots+n_m},\qquad \eta_1\colon I\to V_1
\end{equation}
and are associative and unital.
If we think the components of a non-symmetric operad as expressing $n$-ary operations, substitution can be visualized as in the one-object case \emph{multicategory} composition of \cite[Fig.~2-B]{Leinster:2004a}
\begin{displaymath}
\scalebox{0.7}{
 \begin{tikzpicture}[triangle/.style = {fill=gray!20, regular polygon, regular polygon sides=3}]
\path (0,0) node [triangle,draw,shape border rotate=-90,inner sep=0pt,label=178:$\vdots$] (a) {$\phantom{\theta_n}$}
(0,4) node [triangle,draw,shape border rotate=-90,inner sep=0pt,label=178:$\vdots$] (b) {$\phantom{\theta_n}$}
(3,2) node [triangle,draw,shape border rotate=-90,label=135:$1$,label=230:$m$,label=178:$\vdots$] (c) {$\phantom{\theta_n}$}
(9,2) node [triangle,draw,shape border rotate=-90,inner sep=-25pt,label=178:$\vdots$] (d) {$\phantom{\qquad\theta\circ(\theta_1,\ldots,\theta_n)}$};
\draw [-] (a) .. controls +(right:2cm) and +(left:1cm).. (c.220);
\draw [-] (b) .. controls +(right:2cm) and +(left:1cm).. (c.140);
\draw [-] (d) to  (11.7,2);
\draw [-] (c) to  (4.2,2);
\draw [-] (7.7,3) to node [above] {$1$} (8.2,3);
\draw [-] (7.7,1) to node [below] {$n$} (8.2,1);
\draw [-] (-.85,4.3) to node [above] {$1$} (-.25,4.3);
\draw [-] (-.85,3.6) to node [below] {$n_1$} (-.25,3.6);
\draw [-] (-.85,0.3) to node [above] {$1$} (-.25,0.3);
\draw [-] (-.85,-0.4) to node [below] {$n_m$} (-.25,-0.4);
\node () at (5.5,2) {$\mapsto$};
\end{tikzpicture}}
\end{displaymath}
For examples and more details on these structures, see e.g. \cite[\S~B.7]{Species}, \cite[\S~1.1]{FresseHomotopy} or \cite[\S~5.8]{AlgebraicOperads} where the equivalent classical, partial and combinatorial definitions are provided.

Comonoids in $(\VGrd,\circ,\X)$ is what we call \emph{non-symmetric cooperads}, namely graded objects $\{Z_n\}_{n\geq0}$ equipped with
\begin{equation}\label{eq:nonsymcooperads}
\delta_n\colon Z_n\to\sum_{\substack{m\geq0 \\ n_1+\ldots+n_m=n}}Z_m \ot Z_{n_1}\ot\ldots\ot Z_{n_k},\qquad \epsilon_n=
\begin{cases}
 Z_1\to I, & n=1 \\
Z_n\to0, & n\neq1
\end{cases}
\end{equation}
which are coassociative and counital. Denote the respective categories by $\Mon_\circ(\VGrd)=\VOpd$ and $\Comon_\circ(\VGrd)=\VCoopd$.
In various references the term `cooperad' refers to different things: in \cite[\S~B.7.1]{Species}, non-symmetric operads in the setting of graded vector spaces are defined to be comonoids with respect to a \emph{lax} monoidal structure on $\VGrd$ namely 
\begin{equation}\label{eq:circ'}
(V\circ'W)_n=\prod_{m\geq0}V_m\ot\sum_{n_1+\ldots+n_m=n}W_{n_1}\ot\ldots\ot W_{n_m}
\end{equation}
Due to existence of finite biproducts in $\Vect$, and since the sum is over a finite number of combinations for each fixed $m$, the comultiplication maps $W\to W\circ' W$ can be written as the duals of \cref{eq:nonsymoperad}. However, contrary to graded comonoids \cref{eq:gradedcomonoid} for which the two conventions agreed in the context of vector spaces, these two notions of cooperads are in general distinct: the product over all natural numbers is infinite thus does not coincide with a sum. Still, every $\circ$-comonoid gives rise to a $\circ'$-comonoid in vector spaces (or more generally in a category enriched in pointed sets) via 
\begin{displaymath}
 W_n\xrightarrow{\delta_n}\sum_{m\geq0}W_m \ot\sum_{n_1+\ldots+n_m=n} W_{n_1}\ot\ldots\ot W_{n_k}\hookrightarrow
 \prod_{m\geq0}W_m \ot\sum_{n_1+\ldots+n_m=n} W_{n_1}\ot\ldots\ot W_{n_k}
\end{displaymath}
This also relates to \cite[Rem.~4.25]{VCocats} where any $\ca{V}$-\emph{cocategory} gives rise to a $\ca{V}$-\emph{opcategory}, the former being a comonad in the bicategory of $\ca{V}$-matrices and the latter a $\ca{V}^\op$-category.

\begin{rmk}\label{rmk:positivegraded}
In the full subcategory of \emph{positive} graded objects $\Grd_+(\ca{V})\subset\VGrd$ of $\mathbb{N}$-families $\{V_n\}_{n>0}$ whose $0$-component is initial $V_0=0$, or equivalently families $\{V_n\}_{n\geq1}$, the two notions of cooperads match. In that case, $m$ as well as all $n_i$'s in both formulas \cref{eq:subVGrd,eq:circ'} are non-zero, which forces the number $m$ of possible decompositions of $n$ to actually be bounded $m\leq n$. Therefore, when finite biproducts exist in $\ca{V}$, the two monoidal products coincide $\circ=\circ'$ and consequently, the two notions of non-symmetric positive cooperads do too. 

Positive graded objects as well as non-symmetric positive operads and cooperads are sometimes studied on their own right, e.g. \cite[Def.~1.14]{Operadsinalgtopphy} in arbitrary monoidal categories or \cite[\S~3.1]{WIT} in vertical bicomplexes.
\end{rmk}

In the world of species, in a symmetric monoidal category where $\ot$ preserves colimits, the $m$-th Cauchy tensor product is computed to be
\begin{equation}\label{eq:mCauchySym}
 A^{\bullet m}=\int^{n_1,\ldots,n_m}\ca{P}(n_1+\ldots+n_m,-)*A_{n_1}\otimes\ldots \otimes A_{n_m}
\end{equation}
with explicit components, using \cref{eq:CauchySym} repeatedly,
\begin{align}
A^{\bullet m}_n=&\sum_{n_1=0}^n\sum_{\mathrm{Sh}(n_1,n\mi n_1)}A_{n_1}\ot\sum_{n_2=0}^{n\mi n_1}\sum_{\mathrm{Sh}(n_2,n\mi n_1\mi n_2)}A_{n_2}\ot\ldots
\ot\sum_{n_{m\mi 1}=0}^{n\mi \ldots\mi n_{m\mi 2}}\sum_{\mathrm{Sh}(n_{m\mi 1},n_m)}A_{n_{m\mi 1}}\ot A_{n_m} \nonumber\\
=&\sum_{n_1+\ldots+n_m=n}\sum_{\mathrm{Sh}(n_1,n\mi n_1)}\ldots\sum_{\mathrm{Sh}(n_{m\mi 1},n_m)}A_{n_1}\ot\ldots \ot A_{n_m}=
\sum_{\substack{n_1+\ldots+n_m=n \\ \mathrm{Sh}(n_1,\ldots,n_m)}}A_{n_1}\ot\ldots\ot A_{n_m}
\label{mCauchyKelly}
\end{align}
where $\mathrm{Sh}(n_1,\ldots,n_m)$ is that subset of $P_n$ of permutations $\sigma$ for which 
\begin{displaymath}
\sigma(1)<\ldots<\sigma(n_1),\quad \sigma(n_1+1)<\ldots<\sigma(n_1+n_2), \quad\ldots\quad, \sigma(n\mi n_m+1)<\ldots<\sigma(n)
\end{displaymath}
i.e. the order is preserved only among elements in the same 
$n_i$-part, see e.g. \cite[\S~1.1]{NotesonAlgOperads}.

The substitution monoidal product in species is $A\circ B=\int^mA_m\ot B^{\bullet m}$
as in \cite[~{}3.2]{KellyonMay}, with
\begin{equation} \label{eq:nKellysub}
 (A\circ B)_n=\int^m A_m\ot B^{\bullet m}_n
 \stackrel{\cref{mCauchyKelly}}{\cong}\int^m\sum_{\substack{n_1{+}\ldots{+}n_m=n \\ \textrm{Sh}(n_1,\ldots,n_m)}}A_m\ot B_{n_1}\ot\ldots\ot B_{n_m}
\end{equation}
Analogously to \cref{eq:subwithHadCau}, substitution can be expressed in terms of the pointwise and Cauchy by
\begin{equation}\label{eq:subVSym}
 \begin{tikzcd}[row sep=.3in]
\VP\times\VP\ar[r,dashed,"\circ"]\ar[d,"1\times(-)^\bullet"'] & \VP \\
\VP\times[\ca{P},\VP]\ar[r,"{\wt{\ot}}"'] & {}[\ca{P},\VP]\ar[u,"\mathrm{\displaystyle\int^{m\in\ca{P}}}"']
\end{tikzcd}
 \end{equation}
or equivalently \cite[B.11]{Species} $A\circ B=\colim_{m}(A_m\ot B^{\bullet m})=\sum_{m\geq0}\left(A_k\ot B^{\bullet m}\right)_{P_m}$.
The unit species is the same as \cref{eq:unitsub}, expressed by $\X=\ca{P}(1,-)*I.$
Like substitution for graded objects in $\ca{V}$, this tensor product is not symmetric and moreover, when $\ca{V}$ is closed, $(\VSym,\circ,\X)$ is only left closed via $-\circ A\dashv [A,-]_\ell^\circ$ defined by
\begin{equation}\label{eq:subleftclosed}
 [A,B]^\circ_{\ell,n}
 =\int_m[A^{\bullet n}_m,B_m]_\ell.
\end{equation}

Similarly to the alternative substitution $\circ'$ \cref{eq:circ'} for graded objects in $\ca{V}$, in \cite[\S~B.4.4]{Species} there is the limit counterpart operation on species
\begin{equation}
 A\circ' B=\lim_m (A_m\ot B^{\bullet m})
\end{equation}
which makes $\VSp$ into a lax monoidal category. 
In the context of positive species, namely species $A\colon\ca{B}\to\ca{V}$ for which $A(\emptyset)=0$, the substitution formula becomes \cite[Def.~8.5~(8.8)]{Species}
\begin{equation}\label{eq:positivespeciessub}
 (A\circ B)(I)=\sum_{X\vdash I}A(X)\ot\bigotimes_{S\in X}B(S)
\end{equation}
where the sum is now over all \emph{partitions} $X$ of the finite set $I$, namely disjoint non-empty subsets of $I$ whose union is $I$; see also \cite[\S~1.4]{CombinatorialSpecies}. 
Moreover, as explained in \cite[\S~B.4.8]{Species}, in the case of positive species the two products coincide $\circ=\circ'$ as was the case for positive graded objects earlier.

Monoids in $(\VSp,\circ,\X)$ are \emph{symmetric}, or $\Sigma$-\emph{operads}: they are species $\{A_n\}_{n\geq0}$ equipped with multiplication and unit as in \cref{eq:nonsymoperad}, satisfying associativity and unitality as well as the extra axiom of $P_n$-equivarience, namely naturality of the morphism $\mu\colon A\circ A\to A$ in $\VSp$. Symmetric operads were originally introduced in \cite{Mayoperad} in topological spaces, and have been extensively studied ever since in various contexts. A few suggestive references are \cite{GetzlerJones,AlgebraicOperads,GambinoJoyal}.

\begin{rmk}
Many authors e.g. \cite{MarklOperadsProps} consider operad composition to have components
$$A_m\ot A_{n_1}\ot\ldots A_{n_m}\to A_{n_1+\ldots+n_m}$$
for any $m\geq1$ excluding $m=0$, but for $n_i\geq0$. As also explained for example in \cite[{}~I.1.1.3]{FresseHomotopy}, this is due to the fact that when $m=0$, the multiplication becomes an endomorphism $A_0\to A_0$ which by the axioms is forced to be the identity. So even though it does not provide any extra information, it is required for correctly specifying the axioms.

Also, these should not be confused with \emph{positive} symmetric operads, namely monoids in the full subcategory of species $\{A_n\}_{n\geq1}$  or equivalently $A_0=0$, in which case all the $n_i$'s would furthermore be strictly greater than $0$.
This is the convention used in \cite{Operadsinalgtopphy} in the context of an arbitrary symmetric monoidal category, and in fact matches the original one in \cite{Mayoperad}, since $\ca{V}$ therein is the category of based compactly generated Hausdorff spaces and $A_0$ is set to be $\{*\}$, the zero object. Positive operads are sometimes called \emph{non-unitary} \cite[\S~I.1.1.19]{FresseHomotopy} forming a category denoted $\nc{Op}_\emptyset$.
\end{rmk}

Dually, comonoids in $(\VSp,\circ,\X)$ are what we call \emph{symmetric cooperads}, namely species $\{C_n\}_{n\geq0}$ equipped with comultiplication and counit as in \cref{eq:nonsymcooperads} which are coassociative, counital and $P_n$-equivariant, see e.g. \cite{GetzlerJones}. The categories of symmetric operads and cooperads are denoted by $\Mon_\circ(\Sp(\ca{V}))=\nc{sOpd}(\ca{V})$ and $\Comon_\circ(\Sp(\ca{V}))=\nc{sCoopd}(\ca{V})$, and their positive counterparts $\nc{sOpd}_+$ and $\nc{sCoopd}_+$. Similarly to earlier patterns, a different convention is that a symmetric cooperad is a comonoid in $(\VSp,\circ',\X)$, in which case the comultiplication and counit arrows are
\begin{equation}\label{eq:circ'cooperad}
Z_{n_1+\ldots+n_m}\to Z_m \ot Z_{n_1}\ot\ldots\ot Z_{n_k},\qquad \epsilon=\begin{cases}
\epsilon_1=Z_1\to I, & n=1 \\
\epsilon_n=Z_n\to0, & n\neq0
\end{cases}
\end{equation}
satisfying appropriate axioms. When $\ca{V}$ has a zero object, this definition is precisely dual to that of an operad, namely a cooperad is an operad in $\ca{V}^\op$.

As was the case for positive graded objects and non-symmetric (co)operads, positive species and positive symmetric (co)operads $\{A_n\}_{n\geq1}$ are sometimes considered as the fundamental notions rather than their non-positive counterparts, see e.g. \cite{KoszulDualityOperads,ChingGoodwillie}.

\subsection{Duoidal structures and enrichment}
\label{sec:duoid-struct-enrichm}

In this section, we explore various duoidal structures on graded objects and species on $\ca{V}$. Consequently, applying Sweedler theory for duoidal categories as described in \cref{sec:SweedlerTheoryduoid}, we obtain certain enrichments of the various categories of (co)monoids with respect to the above described tensor products.

The following two examples generalize \cite[Ex.~6.22]{Species} from the category of vector spaces to the context of arbitrary monoidal categories, under certain assumptions. In particular, the following is a corollary to \cref{prop:duoidalpresheaves}.

\begin{lemma}\label{prop:duoidVGrd1}
If $\ca{V}$ is a symmetric monoidal category with (finite) coproducts which are preserved by $\ot$ in both variables, $(\VGrd,\bullet,\1,\otimes,\B{I})$ is a duoidal category. 
\end{lemma}

\begin{proof}[Sketch]
Recall that the pointwise $\ot$ and Cauchy $\bullet$ products for $\ca{V}$-graded objects are as in \cref{eq:Had,eq:CauchySym}. 
The interchange law $(V\ot W)\bullet (U\ot Z)\to (V\bullet U)\ot (W\bullet Z)$ for the suggested duoidal structure is explicitly given by
\begin{align}
 \sum_{k+m=n}V_k\ot W_k\ot U_m\ot Z_m&\hookrightarrow\sum_{k_1+m_1=n}V_{k_1}\ot U_{m_1}\ot\sum_{k_2+m_2=n} W_{k_2}\ot 
Z_{m_2} \label{inter1}\\
&\phantom{A}\cong\sum_{\substack{k_1+m_1=n \\ k_2+m_2=n}}V_{k_1}\ot W_{k_2}\ot U_{m_1}\ot Z_{m_2} \nonumber
\end{align}
which is the natural inclusion of a smaller sum into a larger sum, using symmetry as well as the fact that 
$\otimes$ commutes with coproducts. The rest of the structure maps \cref{eq:deltamuiota} are
\begin{align}
 \delta_{\1}&\colon{\1}\to\1\ot\1\textrm{ by } (\delta_{\1})_0\colon I\cong I\ot I \textrm{ and }\id_0 \textrm{ elsewhere} \label{eq:reststructuresbulletotimes}\\
 \mu_{\I}&\colon \I\bullet\I\to\I\textrm{ by } (\mu_{\I})_n\colon \sum_{k+m=n}I\xrightarrow{\nabla}I \nonumber \\
 \iota&\colon\1\to\I\textrm{ by }\iota_0\colon I\xrightarrow{\id}I\textrm{ and }\iota_n\colon0\xrightarrow{!}I\textrm{ for }n\neq0 \nonumber
\end{align}
\end{proof}

The following result shows that also combined the other way, pointwise and Cauchy form a duoidal structure when binary biproducts and a zero object exist.

\begin{lemma}\label{prop:duoidVGrd1'}
If $\ca{V}$ is a symmetric monoidal category with finite biproducts which are preserved by $\ot$ in both variables, $(\VGrd,\otimes,\B{I},\bullet,\1)$ is a duoidal category.
\end{lemma}

\begin{proof}[Sketch]
 Switching the role of the two monoidal products from the previous proof, the interchange law
$(V\bullet U)\ot (W\bullet Z)\to(V\ot W)\bullet(U\ot Z)$ is given by the surjection
\begin{equation}\label{inter2}
 \sum_{\substack{k_1+m_1=n \\ k_2+m_2=n}}V_{k_1}\ot W_{k_2}\ot U_{m_1}\ot Z_{m_2}\twoheadrightarrow\sum_{k+m=n}V_k\ot W_k\ot U_m\ot Z_m
\end{equation}
which is identity when $k_1=k_2$ and $m_1=m_2$ and the zero map in any other case. Of course this is the same as the projections from a larger (finite) product to a smaller product.
The other structure maps are
\begin{align}
 \delta_{\I}&\colon{\I}\to\I\bullet\I\textrm{ by } (\delta_{\I})_n\colon I\xrightarrow{\Delta}\sum_{k+m=n}I\cong\prod_{k+m=n}I \label{eq:deltaI} \\
 \mu_{\1}&\colon \1\ot\1\to\1\textrm{ by } (\mu_{\1})_0\colon I\ot I\cong I\textrm{ and }(\mu_{\1})_n\colon  0\xrightarrow{\id}0
 \nonumber \\
 \iota&\colon\I\to\1\textrm{ by }\iota_0\colon I\xrightarrow{\id}I, \iota_n\colon I\xrightarrow{!}0 \nonumber
\end{align}
\end{proof}

Now \cref{thm:sweedlerduoidal} applies to provide two distinct enriched categories. The first one concerns $\ot$-monoids and comonoids, using a $\bullet$-monoidal structure.

\begin{thm}\label{thm:CauchyHadGrdEnr}
Suppose $\ca{V}$ is a locally presentable, symmetric monoidal closed category. Then the category
$\Grd(\Mon(\ca{V}))$ of graded objects in monoids is tensored and cotensored enriched in the symmetric monoidal closed category $(\Grd(\Comon(\ca{V})),\bullet,\1)$ of graded objects in comonoids, with the Cauchy tensor product.
\end{thm}

\begin{proof}
For the duoidal category $(\VGrd,\bullet,\1,\otimes,\I)$ of \cref{prop:duoidVGrd1}, we first of all know
that $\VGrd$ is locally presentable. Moreover, in the previous section it was established that $(\VGrd,\bullet,\1)$ is a symmetric monoidal closed category, and also the pointwise product $\ot$ preserves all colimits in both variables since it is symmetric and closed under the above assumptions.  The result now follows from \cref{thm:sweedlerduoidal}.
\end{proof}

Notice how this enrichment is different than the earlier \cref{thm:HadGrdEnr}, since the monoidal base uses a different structure.  Explicitly, $[-,-]^\bullet$ induces an action of $\ot$-comonoids on $\ot$-monoids, and its two adjoints $P^\bullet,\triangleright^\bullet$ form the desired structure as in
\cref{eq:duoidalhomPtriangle}.
In particular, \cref{eq:duoidalconvolution} here gives $[Z,V]^\bullet$ the structure of a graded object in $\Mon(\ca{V})$, for any $Z\in\Grd(\Comon(\ca{V}))$ and $V\in\Grd(\Mon(\ca{V}))$, via
\begin{equation}\label{eq:convHadCauchy}
 \begin{tikzcd}[row sep=.15in]
 \left(\left([Z,V]^\bullet\ot[Z,V]^\bullet\right)\bullet Z\right)_n\ar[d,equal]\ar[dddrr,dashed,bend left] && \\
\displaystyle\sum_{k+m=n}\left(\prod_{i\geq0}[Z_i,V_{i+k}]\ot\prod_{j\geq0}[Z_j,V_{j+k}]\right)\ot Z_m\ar[d,"\sum 1\ot\delta_m"'] && \\
\displaystyle\sum_{k+m=n}\left(\prod_{i\geq0}[Z_i,V_{i+k}]\ot\prod_{j\geq0}[Z_j,V_{j+k}]\right)\ot Z_m\ot Z_m\ar[d,"\cref{inter1}"'] & &\\
\displaystyle\left(\sum_{k_1+m_1=n}\prod_{i\geq0}[Z_i,V_{i+k_1}]\ot Z_{m_1}\right)\ot
\left(\sum_{k_2+m_2=n}\prod_{j\geq0}[Z_j,V_{j+k_2}]\ot Z_{m_2}\right)\ar[r,"\textrm{ev}\ot\textrm{ev}"']
&
V_n\ot V_n\ar[r,"\mu_n"'] & V_n
 \end{tikzcd}
\end{equation}

In the opposite combination, we obtain an enrichment of $\bullet$-monoids in $\bullet$-comonoids, using the monoidal closed $\ot$-structure.

\begin{thm}\label{thm:HadCauchyGrdEnr}
Suppose $\ca{V}$ is a locally presentable, symmetric monoidal closed category with finite biproducts. The 
category $\nc{gMon}(\ca{V})$ of graded monoids is tensored and cotensored enriched the symmetric monoidal closed 
category $(\nc{gComon}(\ca{V}),\ot,\I)$.
\end{thm}

\begin{proof}
For the duoidal category $(\VGrd,\otimes,\B{I},\bullet,\1)$ of \cref{prop:duoidVGrd1'} under these assumptions, \cref{thm:sweedlerduoidal} applies once again, since both structures are symmetric monoidal closed therefore more than sufficient to cover the required conditions. 
\end{proof}

The induced action $[\mi,\mi]^\ot$ of $\bullet$-comonoids on $\bullet$-monoids here produces a $\bullet$-monoid $[Z,V]^\ot$ via \cref{eq:duoidalconvolution} expressed by
\begin{equation}\label{eq:convCauchyHad}
 \begin{tikzcd}[row sep=.15in]
 (([Z,V]^\ot\bullet[Z,V]^\ot)\ot Z)_n\ar[d,equal]\ar[dddrr,dashed,bend left] && \\
 \left(\displaystyle\sum_{k_1+m_1=n}[Z,V]^\ot_{k_1}\ot[Z,V]^\ot_{m_1}\right)\ot Z_n\ar[d,"1\ot\delta_n"'] && \\
 \left(\displaystyle\sum_{k_1+m_1=n}[Z_{k_1},V_{k_1}]\ot[Z_{m_1},V_{m_1}]\right)\ot \displaystyle\sum_{k_2+m_2=n}Z_{k_2}\ot Z_{m_2}\ar[d,"\cref{inter2}"'] && \\
 \displaystyle\sum_{k+m=n}[Z_k,V_k]\ot Z_k\ot[Z_m,V_m]\ot Z_m\ar[r,"\sum\textrm{ev}\ot\textrm{ev}"'] & \displaystyle\sum_{k+m=n} V_k\ot V_m\ar[r,"\mu_n"'] & V_n
 \end{tikzcd}
\end{equation}

Similarly to $\ca{V}$-graded objects, there are two duoidal structures on $\ca{V}$-species, extending \cite[Prop.~8.68]{Species} from vector spaces to arbitrary monoidal categories. Once again, the first one comes from the much more general \cref{prop:duoidalpresheaves}.

\begin{lemma}\label{prop:duoidVSym1}
If $\ca{V}$ is a symmetric monoidal category with coproducts preserved by $\ot$ in both entries, $(\VSym,\bullet,\1,\otimes,\B{I})$ is a duoidal category. If moreover $\ca{V}$ has finite biproducts, $(\VSym,\otimes,\B{I},\bullet,\1)$ is also a duoidal category.
\end{lemma}

\begin{proof}[Sketch]
Recall that the pointwise and Cauchy products for $\ca{V}$-species are as in \cref{eq:Had,eq:CauchySym}. The interchange laws 
\begin{displaymath}
\begin{tikzcd}
 (A\ot B)\bullet (C\ot D)\ar[r,shift left] & (A\bullet C)\ot (B\bullet D)\ar[l, shift left]
\end{tikzcd}
\end{displaymath}
follow from those for graded objects as in \cref{inter1,inter2}, along with the extra shuffle information that determines the position of a 
$k$-element subset inside $\{0,1,\ldots,n\mi1\}$:
\begin{displaymath}
 \begin{tikzcd}
{\displaystyle\sum_{k+m=n}}\;{\displaystyle\sum_{\mathrm{Sh}(k,m)}}A_k\ot B_k\ot C_m\ot D_m\ar[r,shift left,hook] & 
{\displaystyle\sum_{\substack{k_1+m_1=n \\ k_2+m_2=n}}}\;{\displaystyle\sum_{\substack{\mathrm{Sh}(k_1,m_1) \\ \mathrm{Sh}(k_2,m_2)}}}A_{k_1}\ot 
C_{m_1}\ot B_{k_2}\ot D_{m_2}\ar[l,shift left,->>]  
 \end{tikzcd}
\end{displaymath}
The inclusion is the natural one, whereas the surjection is defined by the identity when $k_1=k_2,m_1=m_2,\sigma=\sigma'\in\textrm{Sh}$ and by the 
zero map in any other case, equivalently the canonical projection since the sums are finite. The rest of the structure maps are as in \cref{eq:reststructuresbulletotimes,eq:deltaI}.

Moreover, as shown in \cite[8.58]{Species} for vector species, not only is the pointwise tensor product functor lax and oplax monoidal with respect to the Cauchy tensor product functor, but it is moreover \emph{bilax} monoidal - which is not the case for graded objects.
\end{proof}

\begin{rmk}
As explained in detail in \cite[Ex.~6.78 \& Prop.~8.68]{Species}, when $\ca{V}=\mathbf{Vect}^{f}_k$ the category of finite-dimensional $k$-vector spaces, those pairs of duoidal structures including the pointwise and the Cauchy tensor products are \emph{contragredient} to one another, both for graded objects and species.
\end{rmk}

In \cref{sec:pointw-or-hadam,sec:conv-struct}, it was established how both monoidal structures form a symmetric monoidal closed structure on species. Therefore
once again, \cref{thm:sweedlerduoidal} applies to both duoidal structures of the above proposition, establishing a $\bullet$-enrichment of $\ot$-monoids in $\ot$-comonoids as well as a $\ot$-enrichment of $\bullet$-monoids in $\bullet$-comonoids in the category of species in a locally presentable $\ca{V}$.

\begin{thm}\label{thm:CauchyHadSymEnr}
If $\ca{V}$ is a locally presentable, symmetric monoidal closed category, the category $\Sp(\Mon(\ca{V}))$ of species in monoids is tensored and cotensored enriched in the symmetric monoidal closed category $(\Sp(\Comon(\ca{V})),\bullet,\1)$ of species in comonoids with the Cauchy product. If $\ca{V}$ moreover has finite biproducts, the category $\nc{TwMon}(\ca{V})$ of twisted monoids is tensored and cotensored enriched in the symmetric monoidal closed category $(\nc{TwComon}(\ca{V}),\ot,\I)$ of twisted comonoids with the pointwise product.
\end{thm}

The induced actions of the comonoids on monoids, as well as all related Sweedler Theory functors, are analogous to those for graded objects, e.g. \cref{eq:convHadCauchy,eq:convCauchyHad}.

Moreover, the pointwise product also forms duoidal structures with the substitution product.
The following result generalizes \cite[Ex.~6.23]{Species} which establishes the $(\circ,\otimes)$-duoidal structure for graded 
vector spaces.

\begin{lemma}\label{prop:duoidVGrd2}
If $\ca{V}$ is a symmetric monoidal category with coproducts which are preserved by $\ot$ in both variables, $(\VGrd,\circ,\X,\otimes,\B{I})$ is a duoidal category.
\end{lemma}

\begin{proof}[Sketch]
 Recall that the pointwise and substitution monoidal products of graded objects in $\ca{V}$ are \cref{eq:Had,eq:subVGrd}. The interchange law
 \begin{equation}\label{eq:hardinterchange}
  \begin{tikzcd}[row sep=.15in]
(V\otimes W)\circ (U\otimes Z)\ar[d,phantom,"{{\cong}}"] \\
\displaystyle\sum_{m\geq0}V_m\otimes W_m\ot\displaystyle\sum_{n_1+\ldots+n_m=n}U_{n_1}\ot Z_{n_1}\ot\ldots \ot U_{n_m}\ot Z_{n_m}\ar[d,shift 
left,hook] \\
\displaystyle\sum_{k\geq0}V_{k}\otimes \displaystyle\sum_{n_1+\ldots+n_{k}=n}U_{n_1}\ot\ldots \ot U_{n_{k}} \ot \displaystyle\sum_{\ell\geq0} 
W_{\ell}\ot \displaystyle\sum_{n'_1+\ldots +n'_{\ell}=n} Z_{n'_1}\ot\ldots\ot Z_{n'_{\ell}}
\ar[d,phantom,"\cong"] \\
(V\circ U)\ot(W\circ Z)
  \end{tikzcd}
 \end{equation}
is again the natural inclusion from a smaller to a larger sum, namely identity when $k=\ell=m$ and $n_i=n'_j$, using symmetry and the fact that $\ot$ preserves coproducts.
The rest structure maps are almost identical to \cref{eq:reststructuresbulletotimes} by now using the $\circ$-monoidal unit $\X$ \cref{eq:unitsub} instead of $\1$: 
\begin{align}
 \delta_\X&\colon\X\to\X\ot\X\textrm{ by }(\delta_\X)_1\colon I\cong I\ot I\textrm{ and }\id_0\textrm{ elsewhere}\nonumber \\
 \mu_\I &\colon \I\circ \I\to \I\textrm{ by }(\mu_\I)_n\colon\sum_{\substack{m\geq0 \\ n_1{+}\ldots{+}n_m=n}}I\ot I\xrightarrow{\nabla}I\nonumber \\
 \iota & \colon\X\to\I\textrm{ by }\iota_1\colon I\xrightarrow{\id}I\textrm{ and }\iota_n\colon 0\xrightarrow{!}I\textrm{ for }n\neq1 \label{eq:reststructurescircotimes}
\end{align}
\end{proof}

The above proof naturally extends to the context of $\ca{V}$-species, generalizing \cite[\S~B.6.5]{Species}.

\begin{lemma}\label{prop:duoidVSym2}
 If $\ca{V}$ is a symmetric monoidal category with all colimits preserved by $\ot$ in both entries, then $(\VSym,\circ,\X,\otimes,\B{I})$ is a duoidal category.
\end{lemma}

\begin{proof}[Sketch]
 The relevant structures are as in \cref{eq:Had,eq:nKellysub}. The interchange law now is
 \begin{displaymath}
\begin{tikzcd}[row sep=.15in]
(A\otimes B)\circ (C\otimes D)\ar[d,phantom,"{{\cong}}"] \\
\displaystyle\int^m\sum_{\substack{n_1{+}\ldots{+}n_m=n \\ \mathrm{Sh}(n_1,\ldots,n_m)}}A_m\otimes B_m\ot C_{n_1}\ot D_{n_1}\ot\ldots \ot C_{n_m}\ot D_{n_m}\ar[d,shift 
left,hook] \\
\displaystyle\int^{k,\ell}\sum_{\substack{n_1{+}\ldots{+}n_k=n \\ n'_1{+}\ldots +n'_{\ell}=n}}\sum_{\substack{\mathrm{Sh}(n_1,\ldots,n_k) \\\mathrm{Sh}(n'_1,\ldots,n'_\ell)}}A_k\otimes C_{n_1}\ot\ldots \ot D_{n_k}\ot  
B_{\ell}\ot D_{n'_1}\ot\ldots\ot D_{n'_{\ell}}
\ar[d,phantom,"\cong"] \\
(A\circ C)\ot(B\circ D)
\end{tikzcd}
\end{displaymath}
where symmetry as well as $\ot$ preserving all colimits in both variables is employed. The rest of the structures are similar to \cref{eq:reststructurescircotimes}, namely $\delta_\X$ and $\iota$ are the same whereas $\mu_\I$ is again the universal arrow from the colimit induced by identities.
\end{proof}

\begin{rmk}
 As explained in \cref{sec:subst-mono-prod}, substitution does not form a symmetric monoidal structure for neither graded objects not species, and also it is only left closed \cref{eq:subleftclosed}. As a result, \cref{thm:sweedlerduoidal} does not apply like in the earlier cases; however, we do obtain an enrichment of $\Grd(\Mon(\ca{V}))$ in the \emph{reverse} $\circ$-monoidal category $\Grd(\Comon(\ca{V}))$ and similarly for species.
 
However in particular, this duoidal structure induces a $\ot$-monoidal structure on $\circ$-monoids in both cases, as in \cref{eq:starMon}: $(\nc{Opd},\ot,\I)$ and $(\nc{sOpd},\ot,\I)$ are monoidal categories via
 \begin{displaymath}
A_m\ot B_m\ot A_{n_1}\ot B_{n_1}\ot\ldots\ot A_{n_m}\ot B_{n_m}\cong A_m\ot A_{n_1}\ot\ldots\ot A_{n_m}\ot B_m\ot B_{n_1}\ot\ldots\ot B_{n_m}\to A_{n_1+\ldots+n_m}\ot B_{n_1+\ldots+n_m}
 \end{displaymath}
see also e.g. \cite[\S~5.3.3]{AlgebraicOperads}.
\end{rmk}

At this point, the `symmetry' between the duoidal structures that arise ceases to hold. More explicitly, one would expect that when $\ca{V}$ has finite biproducts, $(\VGrd,\otimes,\B{I},\circ,\X)$ is also a duoidal category, based on the similar \cref{prop:duoidVGrd1'} and the expression of substitution in terms of the pointwise and Cauchy product. However, even though there does exist an associative interchange law in the other direction of \cref{eq:hardinterchange} namely the surjection which is identity on the same-degree components and $0$ everywhere else, the structure that fails to extend to this setting is that of $\delta_\I$, see \cref{eq:deltaI}. More precisely, the natural candidate 
\begin{equation}\label{eq:deltafail}
\delta_\I\colon\I\to\I\circ\I=\sum_{m\geq0}\;\sum_{n_1+\ldots+n_m=n}I
\end{equation} 
does \emph{not} land on a finite sum anymore, which would then be identified with a finite product due to the existence of biproducts.

As a result, for establishing the final $(\otimes,\circ)$-duoidal structure which is in fact necessary for an envisioned enrichment of operads in cooperads, we need to restrict to the setting of positive graded objects and species, where the above infinite sum reduces to a finite one, see \cref{rmk:positivegraded}. The following result generalizes the `contragredient' construction of \cite[Prop.~B.31]{Species} for vector species.

\begin{lemma}\label{prop:duoidVGrd+}
 If $\ca{V}$ is a symmetric monoidal category with finite biproducts which are preserved by $\ot$ in both variables, then $(\VGrd_+,\ot,\I,\circ,\X)$ and $(\VSp_+,\ot,\I,\circ,\X)$ are duoidal categories.
\end{lemma}

\begin{proof}[Sketch]
Recall that positive $\ca{V}$-graded objects are those $\{P_n\}_\mathbb{N}$ where $P_n=0$, the initial object. In that case, as discussed earlier, the infinite sum of the substitution formula \cref{eq:subVGrd} for any $m\geq0$ reduces to a finite one since now $m>n$ could only return zeros from the decomposition of $n$ into a sum of non-zero natural numbers:
\begin{displaymath}
 (P\circ Q)_n=\sum_{\substack{1\leq m\leq n \\ n_1{+}\ldots {+}n_m{=}n}}P_m\ot Q_{n_1}\ot\ldots\ot Q_{n_m}\cong\prod_{\substack{1\leq m\leq n \\ n_1{+}\ldots {+}n_m{=}n}}P_m\ot Q_{n_1}\ot\ldots\ot Q_{n_m}=(P\circ'Q)
\end{displaymath}
due to the existence of finite biproducts and the fact that $\ot$ preserves them. Now the interchange law $(P\circ Q)\ot(R\circ S)\to(P\ot R)\circ (Q\ot S)$ is the analogous surjection as in \cref{inter2}, or equivalently the surjection of a larger product onto a smaller one:
\begin{equation}\label{eq:interpos}
{\prod^{\substack{1\leq k\leq n \\ 1\leq \ell\leq n}}_{\substack{ n_1{+}\ldots{+}n_k=n \\  n'_1{+}\ldots{+}n'_\ell=n}}}
P_k\otimes Q_{n_1}\ot\ldots\ot Q_{n_k}\ot R_{\ell}\ot S_{n'_1}\ot\ldots \ot S_{n'_\ell}{\twoheadrightarrow}
{\prod_{\substack{1\leq m\leq n \\ n_1{+}\ldots{+}n_m=n}}}
P_m\otimes R_m\ot Q_{n_1}\ot S_{n_1}\ot\ldots\ot Q_{n_m}\ot S_{n_m}
\end{equation}
The rest of the structure maps are dual to \cref{eq:reststructurescircotimes}, namely 
\begin{align}
 \mu_\X &\colon \X\ot \X\to \X\textrm{ by }(\mu_\X)_1\colon I\ot I\cong I \textrm{ and }\id_0\textrm{ elsewhere}\nonumber \\
 \delta_\I & \colon \I\to \I\circ\I\textrm{ by } I\xrightarrow{\Delta}\prod_{\substack{1\leq m\leq n \\ n_1{+}\ldots{+}n_m=n}}I\ot I\nonumber \\
  \iota & \colon\I\to\X\textrm{ by }\iota_1\colon I\xrightarrow{\id}I\textrm{ and }\iota_n\colon I\xrightarrow{!}0\textrm{ for }n\neq1 \label{eq:reststructuresotimescirc}
\end{align}

In the case of positive $\ca{V}$-species, using the formula \cref{eq:positivespeciessub} and the existence of finite biproducts, we can also write the interchange law and the rest of structure maps very similarly to the above ones.
\end{proof}

Following the pattern, having established a $(\ot,\circ)$-duoidal structure on positive graded objects and species, there is an enrichment of $\circ$-monoids in $\circ$-comonoids namely of positive operads in positive coperads. As discussed earlier, positive symmetric and non-symmetric operads are sometimes taken as the main concept of operads, an approached followed for example in \cite{Operadsinalgtopphy,ChingGoodwillie}.

\begin{thm}\label{thm:HadSubGrdEnr}
Suppose that $\ca{V}$ is a locally presentable, symmetric monoidal closed category with finite biproducts. The category $\nc{Opd}_+(\ca{V})$ of positive operads in $\ca{V}$ is tensored and cotensored enriched in the symmetric monoidal closed category $(\nc{Coopd}_+(\ca{V}),\ot,\I)$ of positive cooperads in $\ca{V}$ with the pointwise tensor product.  Moreover, $\nc{sOpd}_+(\ca{V})$ of positive symmetric operads is tensored and cotensored enriched in $(\nc{sCoopd}_+(\ca{V}),\ot,\I)$ of positive cooperads.
\end{thm}

\begin{proof}
Both $\VGrd$ and $\VSp$ are locally presentable when $\ca{V}$ is.
For the $(\ot,\circ)$-duoidal structure described in \cref{prop:duoidVGrd+} under these assumptions, the pointwise product is symmetric monoidal closed. Moreover, even if the substitution is only left closed \cref{eq:subclosedVGrd,eq:subleftclosed}, it does preserve filtered colimits in both variables: $\mi\circ V$ because it has a right adjoint $[V,\mi]^\circ$, and also $V\circ\mi$ due to the fact that the endofunctor $W\mapsto W^{\bullet m}$ preserves filtered colimits 
\begin{displaymath}
 V\circ\colim_{j\in J}W_j=\colim_{m}V_m\ot (\colim_{j\in J}W_j)^{\bullet m}\cong \colim_{m}V_m\ot\colim_{j\in J}(W_j^{\bullet m})\cong\colim_{j\in J}(V\circ W_j)
\end{displaymath}
for a filtered category $J$, for both graded objects and species and in fact not necessarily positive. 

All conditions of \cref{thm:sweedlerduoidal} are satisfied and the result follows.
\end{proof}
In particular, if $W,Z\in\VCoopdp$, then $W\ot Z\in\VCoopdp$ via
\begin{displaymath}
\begin{tikzcd}[row sep=.1in]
(W\ot Z)_n=W_n\ot Z_n\ar[r,"\delta_n\ot\delta_n"]\ar[ddr,dashed,bend right] & \displaystyle\sum_{\substack{m\geq0\\n_1{+}\ldots{+}n_m{=}n}}W_m\ot W_{n_1}\ot\ldots\ot W_{n_m}\ot\sum_{\substack{k\geq0\\n'_1{+}\ldots{+}n'_k{=}n}}Z_k\ot Z_{n'_1}\ot\ldots\ot Z_{n'_k}\ar[d,"\cref{eq:interpos}"] \\
& \displaystyle\sum_{\substack{\ell\geq0 \\ n_1{+}\ldots{+}n_\ell=n}}W_\ell\ot Z_\ell\ot W_{n_1}\ot Z_{n_1}\ot\ldots\ot W_{n_\ell}\ot Z_{n_\ell}\ar[d,equal] \\
& \left((W\ot Z)\ot(W\ot Z)\right)_n
\end{tikzcd}
\end{displaymath}
The functors of Sweedler Theory in this context are
\begin{gather*}
[\mi,\mi]^\ot\colon\VCoopdp^\op\times \VOpdp\to \VOpdp \nonumber\\
P^\ot\colon\VOpdp^\op\times\VOpdp\to\VCoopdp \nonumber \\
\triangleright^\ot\colon \VCoopdp\times \VOpdp\to \VOpdp
\end{gather*}
where for example, the induced action $[\mi,\mi]^\ot$ of positive cooperads on positive operads signifies that if $Z\in\VCoopdp$ and $V\in\VOpdp$, $[Z,V]^\ot\in\VOpdp$ via \cref{eq:duoidalconvolution} which here becomes 
\begin{equation}\label{eq:convSubHad}
 \begin{tikzcd}[row sep=.1in,column sep=0.2in]
 (([Z,V]^\ot\circ[Z,V]^\ot)\ot Z)_n\ar[d,equal]\ar[ddddr,dashed,bend left=40] & \\
 \left(\displaystyle\sum_{\substack{m\geq0 \\ n_1{+}\ldots{+}n_m{=}n}}[Z,V]^\ot_{m}\ot[Z,V]^\ot_{n_1}\ot\ldots\ot[Z,V]^\ot_{n_m}\right)\ot Z_n\ar[d,"1\ot\delta_n"'] & \\
 \left(\displaystyle\sum_{\substack{m\geq0 \\ n_1{+}\ldots{+}n_m{=}n}}[Z_{m},V_{m}]\ot[Z_{n_1},V_{n_1}]\ot\ldots\ot[Z_{n_m},V_{n_m}]\right)\ot \displaystyle\sum_{\substack{k\geq0 \\ n'_1{+}\ldots{+}n'_{k}}}Z_{k}\ot Z_{n'_1}\ot\ldots\ot Z_{n'_k}\ar[d,"\cref{eq:interpos}"'] & \\
 \displaystyle\sum_{\substack{\ell\geq0 \\ n_1{+}\ldots{+}n_\ell{=}n}}[Z_\ell,V_\ell]\ot Z_\ell\ot[Z_{n_1},V_{n_1}]\ot Z_{n_1}\ot\ldots\ot[Z_{n_\ell},V_{n_\ell}]\ot Z_{n_\ell}\ar[d,"\sum\textrm{ev}\ot\ldots\ot\textrm{ev}"'] & \\
 \displaystyle\sum_{\substack{\ell\geq0 \\ n_1{+}\ldots{+}n_\ell{=}n}} V_\ell\ot V_{n_1}\ot\ldots V_{n_\ell}\ar[r,"\mu_n"'] & V_n
 \end{tikzcd}
\end{equation}
which is analogous to the action of graded comonoids on graded monoids \cref{eq:convCauchyHad}. For the symmetric case, the formulas are analogous using either the coend with shuffles formulation in the finite case
or the partition formulation for finite sets. Notice that this also agrees, in the appropriate context, with the \emph{convolution operad structure} of \cite[3]{HomotopyOperads}.

\begin{rmk}
Notice that even though the above structures refer to positive graded objects and species, namely where $V_0=0$ or equivalently all formulas are written for $n>0$, in particular \cref{eq:convSubHad} can be seen to work in the same way for general (non-positive) graded objects. As mentioned above, there is indeed a well-behaved interchange law using zero maps for the surjection of the infinite sum version of \cref{eq:interpos}. 

The reason is that in fact, it can be verified
that even if $(\VGrd,\ot,\circ)$ fails to be duoidal essentially due to the absence of the counit structure map \cref{eq:deltafail}, the functor $\ot$ has a $\circ$-oplax monoidal structure via the interchange law and $\mu_\X$ as above. As a result, following the proof of \cref{prop:duoidalaction} we do obtain a functor
$$
[\mi,\mi]^\ot\colon\VCoopd^\op\times\VOpd\to\VOpd
$$
whose explicit operadic structure on $[Z,V]^\ot$ is as in \cref{eq:convSubHad}, and moreover \cref{thm:measuduoidal} still applies in the same way establishing its adjoints. Notice that regarding \cref{thm:sweedlerduoidal}, the absence of $\delta_\I$ prevents $\I$ from being a $\circ$-comonoid namely a cooperad in $\ca{V}$, hence the pointwise product does \emph{not} form a monoidal structure for $\VCoopd$ therefore an enrichment cannot be materialized.

In conclusion, in this case the three first functors of Sweedler Theory \cref{thm:measuduoidal} can still be realized in the context of general, symmetric and non-symetric, operads and cooperads in a symmetric monoidal closed $\ca{V}$ which is locally presentable and has a zero object.
\end{rmk}

\begin{rmk}
There is a closely related project sketched e.g. in \cite{Aneltalk}, concerning an enrichment of operads in cooperads with the pointwise monoidal structure. The methodology differs, in that the authors consider categories of algebras for a cocommutative Hopf operad in $\ca{V}$. it is of course expected that there should be formal connections between these these two developments; we leave such considerations for future work at a later, more definite stage of the project in question.
\end{rmk}

Finally, we exhibit an enrichment of the category of symmetric operads (not necessarily positive) in
the category of symmetric cooperads that arised from a certain duoidal structure
in the category of species. This duoidal structure was constructed on
$\Sp(\Set)$ by~\cite{Commutativity} and remarked that their results remains
valid for a locally presentable cartesian closed category $\mathcal{V}$, instead
of $\Set$.

The two monoidal structures considered on $\Sp(\mathcal{V})$ are the
substitution monoidal structure and the convolution of the product, denoted by
$*$, of which we have said a few words at the end of Section~\ref{sec:conv-struct}. A construction of the interchange law
$(A\circ B)*(C\circ D)\to(A*C)\circ(B*D)$ can be found
in~\cite[Prop.~44]{Commutativity}.

\begin{thm}\label{thm:starSubSymEnr}
  Let $\mathcal{V}$ be locally presentable cartesian closed category. The
  category of symmetric operads in $\mathcal{V}$ is enriched in the monoidal
  category of symmetric cooperads in $\mathcal{V}$.
\end{thm}

\begin{proof}
  We have to verify Theorem~\ref{thm:sweedlerduoidal}'s hypotheses for the
  duoidal category $(\Sp(\mathcal{V}),\star,\mathbf{X},\circ,\mathbf{X})$ (both
  tensor products have the same unit, making this a normal duoidal
  category~\cite{Commutativity}). The substitution tensor product can be written
  in terms of coends, the Cauchy tensor product and pointwise tensor products
  (in this case, these are cartesian products), as seen
  in~\eqref{eq:subVSym}. Therefore, substitution is accessible
  by~\cite[Prop~2.4.5]{MakkaiPare}, since it is a small colimit of accessible
  functors. The convolution tensor product $*$ is symmetric and closed. With
  just this we have completed the proof.
\end{proof}

\appendix

\section{Pointwise monoidal structures and enrichment}
\label{sec:pointw-mono-struct}

This section addresses the relationship between the enrichment of categories of
monoids of a functor category equipped with the pointwise tensor product. We
refer to Section~\ref{sec:pointw-or-hadam} the notations used in the theorem
below.

\begin{thm}
  \label{thm:2}
  Let $\mathcal{V}$ be a locally presentable braided monoidal closed category.
  For any small category $\mathcal{J}$, the category there exists a tensored and
  cotensored $\Comon[\mathcal{J},\mathcal{V}]$-category $\mathsf{M}$ with
  underlying category $\Mon[\mathcal{J},\mathcal{V}]$. Furthermore, for each
  $j\in \mathcal{J}$, there is a full and faithful $\Comon(\mathcal{V})$-functor
  $(E_j)_*\mathsf{M}\to\mathsf{A}$ given by on objects by
  \begin{equation}
    \label{eq:13}
    (A\in\Mon[\mathcal{J},\mathcal{V}])\mapsto (A(j)\in\Mon(\mathcal{V})).
  \end{equation}
\end{thm}
\begin{proof}
  The category $[\mathcal{J},\mathcal{V}]$ is locally presentable by
  \cite[Thm.~5.1.6]{MakkaiPare} plus cocompletness. Its tensor product is
  accessible since its composition with each evaluation functor $E_j$ is
  accessible, and these are jointly conservative
  \cite[Prop.~2.4.10]{MakkaiPare}. We can apply Theorem~\ref{thm:STmoncats} to
  $[\mathcal{J},\mathcal{V}]$ to obtain the enriched category $\mathsf{M}$ of
  the statement.

  The internal hom in $[\mathcal{J},\mathcal{V}]$ is defined pointwise from that
  of $\mathcal{V}$. In other words, the functors $E_j$ are strict closed: $[A,B](j)=[A(j),B(j)]$.
  
  We wish to exhibit an isomorphism between the two functors
  $\mathbf{Mon}[\mathcal{J},\mathcal{V}]^{\mathrm{op}}\to\Comon(\mathcal{V})$ in
  the following diagram.
  \begin{equation}
    \label{eq:113}
    E_j\mathsf{M}(-,B)\cong
    \mathsf{A}(-,B(j))E_j
  \quad\qquad
    \begin{tikzcd}[column sep=large,row sep=large]
      \Comon[\mathcal{J},\mathcal{V}]
      \ar[r,"{[\mi,B]}",shift left = 7pt]
      \ar[r,phantom,"\perp"]
      \ar[d,"E_j"']
      &
      \Mon[\mathcal{J},\mathcal{V}]^{\mathrm{op}}
      \ar[d,"E_j^{\mathrm{op}}"]
      \ar[l,"{\mathsf{M}(-,B)}",shift left = 7pt]
      \\
      \Comon(\mathcal{V})
      \ar[r,"{[\mi,B(j)]}",shift left = 7pt]
      \ar[r,phantom,"\perp"]
      &
      \Mon(\mathcal{V})^{\mathrm{op}}
      \ar[l,"{\mathsf{A}(-,B(j))}",shift left = 7pt]
    \end{tikzcd}
  \end{equation}
  The vertical functon on the left,
  $E_j\colon\Comon[\mathcal{J},\mathcal{V}]\to\Comon(\mathcal{V})$, has a left
  adjoint $D_j$ given by $(D_j(c))(j')=0$ if $j\neq j'$ and $(D_j(b))(j)=c$, for
  a comonoid $c\in\mathcal{V}$.
  Similarly, the vertical functor on the right, $E_j^{\mathrm{op}}$ has a left
  adjoint $L_j$ given by $(L_j(a))(j')=1$ if $j\neq j'$ and $(L_j(a))(j)=a$, for
  a monoid $a\in\mathcal{V}$.
  It is clear that $L_j[c,B(j)]\cong [D_j(c),B]$, both taking the value
  $[c,B(j)]$ at $j\in\mathcal{J}$ and the value $1$ at all other objects. We
  deduce that their right adjoints are isomorphic, which is what we wanted.

  The construction of the isomorphism~\eqref{eq:113} ensures that it is
  compatible with the units and counits of the adjunctions depicted horizontally
  in the diagram, in the sense that it
  makes $E_j$ a pseudomorphism of adjunctions.

In order to show that $E_j$ gives rise to an enriched functor, we have to show
the commutativity of the folowing diagrams.
  \begin{equation}
    \begin{tikzcd}
      E_j\mathsf{M}(B,C)\otimes E_jM(A,B)
      \ar[d,"\cong"']\ar[r,"E_j(\mathrm{comp})"]
      &
      E_j\mathsf{M}(A,C)\ar[d,"\cong"]
      \\
      \mathsf{A}(B(j),C(j))\otimes \mathsf{A}(A(j),B(j))
      \ar[r,"\mathrm{comp}"]
      &
      \mathsf{A}(A(j),C(j))
    \end{tikzcd}
    \qquad
    \begin{tikzcd}
      E_jI\ar[r,"\mathrm{id}"]\ar[d,"1"']
      &
      E_j\mathsf{M}(A,A)\ar[d,"\cong"]
      \\
      I\ar[r,"\mathrm{id}"]
      &
      \mathsf{A}(A(j),A(j))
    \end{tikzcd}
  \end{equation}
  These are equivalent to the following diagrams, respectively, since the
  respective horizontal morphisms are obtained by composing with the unit and
  counit of the adjunctions, and these are compatible with the vertical
  isomorphisms by the comments above.
  \begin{equation}
    \begin{tikzcd}
      E_j[G\otimes H,A]\ar[r,"\cong"]
      \ar[d,"\cong"']
      &
      E_j[H,[G,A]]\ar[d,"\cong"]
      \\
      {[G(j)\otimes H(j),A(j)]}\ar[r,"\cong"]
      &
      {[H(j),[G(j),A(j)]]}
    \end{tikzcd}
    \qquad
    \begin{tikzcd}
      E_j[I,A]\ar[r,"\cong"]\ar[d,"\cong"']
      &
      E_jA\ar[d,"1"]
      \\
      {[i,A(j)]}\ar[r,"\cong"]
      &
      A(j)
    \end{tikzcd}
  \end{equation}
  Finally, the commutativity of the diagrams above is clear, from the fact that
  $E_j$ is strict monoidal and strict closed.
\end{proof}

\section{Duoidal structures on categories on functors}
\label{sec:duoid-struct-categ}

Let us denote by $\mathsf{Mult}$ the 2-category of multicategories. Under
certain circumstances, the 2-functor $\mathbb{A}\times\mi$ has a right adjoint
$(\mi)^{\mathbb{A}}$; i.e., the multicategory $\mathbb{A}$ is exponentiable. For
example, it suffices that $\mathbb{A}$ should be representable. A necessary and
sufficient condition is given in~\cite[Prop~2.8]{MR3263282} ($\mathbb{A}$ should
be promonoidal), from where we recall below the description of the
exponential. Details on representable multicategories can be found
in~\cite{RepresentableMulticats}.

We shall write $\mathbb{A}_u$ for the category of unary morphisms of the
multicategory $\mathbb{A}$.
If the exponential $\mathbb{B}^{\mathbb{A}}$ exists, its objects are functors
$\mathbb{A}_u\to\mathbb{B}_u$. A multimap $F_1,\dots,F_n\to G$ in
$\mathbb{B}^{\mathbb{A}}$ is a family of functions
\begin{equation}
  \label{eq:29}
  \alpha(a_1,\dots,a_n;b)\colon
  \mathbb{A}(a_1,\dots,a_n;b)\to
  \mathbb{B}(F_1(a_1),\dots,F_n(a_n);G(b))
\end{equation}
natural in each $b,a_i\in\mathbb{A}_u$. The composition of multimaps is defined
using the composition in $\mathbb{B}$, in the only reasonable way.

\begin{lemma}
  \label{l:4}
  If the multicategory $\mathbb{A}$ is exponentiable, then the 2-functor
  $(\mi)^{\mathbb{A}}\colon \mathsf{Mult}\to\mathsf{Mult}$ is lax monoidal (with
  respect to the cartesian product).
\end{lemma}
\begin{proof}
  The (unique) comonoid structure of $\mathbb{A}$ makes $\mathbb{A}\times\mi$
  into a strong monoidal functor. Then, its right adjoint is lax monoidal.
\end{proof}
The monoidal constraints in the lemma above are the
$\mathbb{B}^{\mathbb{A}}\times\mathbb{C}^{\mathbb{A}}\to
(\mathbb{B}\times\mathbb{C})^\mathbb{A}$ which send some $(F,G)$ to the functor $a\mapsto
(F(a),G(a))$ similarly to \cref{eq:canonical}.

By a monoidal multicategory we mean a pseudomonoid in $\mathsf{Mult}$. This is
a multicategory $\mathbb{M}$ equipped with multicategory functors
$P\colon\mathbb{M}\times\mathbb{M}\to\mathbb{M}\leftarrow \mathbf{1}\colon J$
(where $\mathbf{1}$ is the terminal multicategory), together with invertible
natural transformations $P(P\times 1)\cong P(1\times P)$,
$P(J\times 1)\cong 1\cong P(1\times J)$ satisfying axioms analogous to those of
the definition of monoidal category. In fact, this is equivalent to requiring
that these natural transformations should make the category $\mathbb{M}_u$ of
unary morphisms into a monoidal category.

\begin{cor}
  \label{cor:2}
  The 2-functor $(\mi)^{\mathbb{A}}$ preserves monoidal multicategory
  structures. If $\mathbb{B}$ is a monoidal multicategory, then so is
  $\mathbb{B}^{\mathbb{A}}$, with tensor product defined pointwise.
\end{cor}

Let us denote by $\nc{MonCat}_\ell$ the 2-category of monoidal categories and lax
monoidal morphisms. The fully faithful 2-functor $\nc{MonCat}_\ell\to\mathsf{Mult}$
that sends a monoidal category to its associated (representable) multicategory
preserves cartesian products. Therefore, it preserves pseudomonoids,
establishing an identification between duoidal categories and monoidal representable
multicategories.

As mentioned before, if a multicategory $\mathbb{A}$ is represented by a
monoidal category $\mathcal{A}$, then $\mathbb{B}^{\mathbb{A}}$
exists. Futhermore, \cite[Prop.~2.12]{MR3263282} shows that when $\mathbb{B}$ is
represented by a cocomplete monoidal category $\mathcal{B}$ (by which we mean
that $\mathcal{B}$ is cocomplete and its tensor product is cocontinuous in each variable), and $\mathcal{A}$
is small, then
$\mathbb{B}^{\mathbb{A}}$ is representable by the usual functor category
$[\mathcal{A},\mathcal{B}]$ equipped with the convolution monoidal structure \cref{eq:dayconv}.

Due to its expression as a Kan extension, the convolution of two
functors $A,B\colon\mathcal{A}\to\mathcal{C}$ is another such functor $A\conv
B$, equipped with a universal natural transformation
$\eta_{a,b}^{A,B}\colon A(a)\otimes B(b)\to (A\conv B)(a\otimes b)$.
Here the universality means that any transformation
$A(a)\otimes B(b)\to C(a\otimes b)$ factors through a unique transformation
$A\conv B\Rightarrow C$.

\begin{thm}
  \label{thm:3}
  Let $(\mathcal{A},\otimes,k)$ be a small monoidal category and
  $(\mathcal{B},\diamond,i,\square,j)$ a duoidal category
  (with the notation of Definition~\ref{def:duoidal}). Assume that
  $(\mathcal{B},\diamond, i)$ is a cocomplete monoidal category.
  Then, $[\mathcal{A},\mathcal{B}]$
  has a duoidal structure $(\conv,J,\square,j)$ where $(\square,j)$ is the
  pointwise monoidal structure and $(\conv,J)$ is the convolution of
  $(\otimes,k)$ with $(\diamond,i)$.
\end{thm}
\begin{proof}
  Let $\mathbb{A}$ and $\mathbb{B}$ be the multicategories associated to
  $(\mathcal{A},\otimes,k)$ and $(\mathcal{B},\diamond,i)$. The structure
  $(\square,j)$ is lax monoidal with respect to $(\diamond,i)$, giving rise to a
  monoidal multicategory $\mathbb{B}$, with tensor $\square$. 
  The monoidal
  structure on the representable multicategory $\mathbb{B}^{\mathbb{A}}$ is given pointwise by
  $\square$; see Corollary~\ref{cor:2}. The corresponding pseudomonoid structure
  on $([\mathcal{A},\mathcal{B}],\conv,J)$ in $\nc{MonCat}_\ell$ is the pointwise
  tensor product given by $\square$. We have obtained the required duoidal
  category.
\end{proof}
One can explicitly describe the interchange law
\begin{equation}
  \label{eq:30}
  (A\square B)\conv(C\square D)
  \Rightarrow(A\conv C)\square (B\conv D)
\end{equation}
as the natural transfomation that corresponds to the natural transformation that
follows, where $\zeta$ denotes the intechange law of $\mathcal{B}$.
\begin{equation}
  \label{eq:31}
\begin{tikzcd}[column sep=.1in]
 (A\square B)(a)\diamond(C\square D)(b)
  =
  \bigl(A(a)\square B(a)\bigr)\diamond \bigl(C(b)\square D(b)\bigr)\ar[r,"\zeta"] &
  \bigl(A(a)\diamond C(b)\bigr)\square \bigl(B(a)\diamond D(b)\bigr)\ar[d,"\eta_{a,b}^{A,C}\otimes\eta_{a,b}^{B,D}"] \\
&  (A\conv C)(a\otimes b)\otimes (B\conv D)(a\otimes b). 
\end{tikzcd}
\end{equation}
\begin{cor}
  \label{cor:3}
  Let $(\mathcal{A},\otimes, k)$ be a small monoidal category and
  $(\mathcal{B},\square,i,\gamma)$ a braided monoidal category. Then,
  $[\mathcal{A},\mathcal{B}]$ has a duoidal structure $(\conv,J,\square,i)$
  where $(\square,i)$ is the pointewise monoidal structure and $(\conv,J)$ is
  the convolution of $(\otimes,k)$ with $(\square,i)$.
\end{cor}
In the case of the corollary, the interchange law~\eqref{eq:30} corresponds to
the following transformation.
\begin{equation}\label{eq:31a}
\begin{tikzcd}[column sep=.7in]
 A(a)\square B(a)\square C(b)\square D(b)\ar[dr,dashed,bend right=5]\ar[r,"1\otimes\gamma_{B(a),C(b)}\otimes 1"] &
 A(a)\square C(b)\square B(a)\square D(b)\ar[d,"\eta_{a,b}^{A,C}\otimes\eta_{a,b}^{B,D}"] \\
 & (A\conv C)(a\otimes b)\otimes (B\conv D)(a\otimes b).
 \end{tikzcd}
\end{equation}

\printbibliography

\end{document}